\documentclass[12pt]{amsart}
\usepackage{amssymb,latexsym}
\usepackage{enumerate}
\usepackage{dsfont}
\usepackage{a4wide}
\usepackage{mathrsfs}
\usepackage[cp1250]{inputenc}
\usepackage[T1]{fontenc}
\usepackage{tabularx}
\usepackage{multirow}
\usepackage{scalerel}
\usepackage{etoolbox}
\patchcmd{\section}{\normalfont}{\normalfont\Large}{}{}
\patchcmd{\section}{\scshape}{\bfseries}{}{}
\makeatletter
\renewcommand{\@secnumfont}{\bfseries}
\makeatother
\usepackage{lmodern}
\usepackage{hyperref}
\hypersetup{hypertexnames=false}
\let\originalforall=\forall
\renewcommand{\forall}{\mathop{\vcenter{\hbox{\Large$\originalforall$}}}}

\let\originalexists=\exists
\renewcommand{\exists}{\mathop{\vcenter{\hbox{\Large$\originalexists$}}}}

\newcommand{\free}[2]{#1 \hspace{1.5 pt} \raisebox{1pt}{\scaleobj{0.7}{\circ}}\hspace{1.5 pt} #2}
\makeatletter
\def\subsection{\@startsection{subsection}{3}%
  \z@{.7\linespacing\@plus.7\linespacing}{.7\linespacing}%
  {\normalfont\large\bfseries}}
\makeatother

\newtheorem{thm}{Theorem}[section]
\newtheorem{cor}[thm]{Corollary}
\newtheorem{lem}[thm]{Lemma}
\newtheorem{prop}[thm]{Proposition}
\newtheorem{f}[thm]{Fact}

\newtheorem{de}[thm]{Definiton}
\newtheorem{rem}[thm]{Remark}

\title{Spectral theory of Fourier--Stieltjes algebras}
\author{Przemys\l aw Ohrysko}
\address{Institute of Mathematics of the Polish Academy of Sciences, ul. \'{S}niadeckich 8, 00-656 Warsaw, Poland}
\email{pohrysko@impan.pl}
\thanks{The research of the first author has been supported by National Science Centre, Poland grant no. 2014/15/N/ST1/02124. The second author was partially supported by the National Science Centre (NCN) grant no.~2016/21/N/ST1/02499.}
\author{Mateusz Wasilewski}
\address{Institute of Mathematics of the
Polish Academy of Sciences, ul. \'{S}niadeckich 8, 00-656 Warsaw, Poland}
\email{mwasilewski@impan.pl}
\begin{document}
\baselineskip=21pt
\begin{abstract}
In this paper we start studying spectral properties of the Fourier--Stieltjes algebras, largely following Zafran's work on the algebra of measures on a locally compact group. We show that for a large class of discrete groups the Wiener--Pitt phenomenon occurs, i.e. the spectrum of an element of the Fourier--Stieltjes algebra is not captured by its range. We also investigate the notions of absolute continuity and mutual singularity in this setting; non-commutativity forces upon us two distinct versions of support of an element, indicating a crucial difference between this setup and the realm of Abelian groups. In spite of these difficulties, we also show that one can introduce and use generalised characters to prove a criterion on belonging of a multiplicative-linear functional to the Shilov boundary of the Fourier--Stieltjes algebra.
\end{abstract}



\maketitle

\raggedbottom
\makeatletter
\let\origsection\section
\renewcommand\section{\@ifstar{\starsection}{\nostarsection}}

\newcommand\nostarsection[1]
{\sectionprelude\origsection{#1}\sectionpostlude}

\newcommand\starsection[1]
{\sectionprelude\origsection*{#1}\sectionpostlude}

\newcommand\sectionprelude{%
  \vspace{1.5em}
}

\newcommand\sectionpostlude{%
  \vspace{1.5em}
}
\makeatother
\section{Introduction}

In this article we extend the study of spectral properties of the convolution algebra of measures on a locally compact Abelian group to the non-commutative setting. Before describing our results, let us present a brief exposition of this classical subject, which is the main motivation for our work.
\subsection{Overview of the subject}
It is clear that for a measure $\mu$ on a locally compact group $G$ the (closure of the) range of the Fourier transform of $\mu$ is contained in its spectrum. However, as it was observed by N. Wiener and H. R. Pitt \cite{wp} this inclusion may be proper and in recognition of the authors this strange spectral behaviour was called the \textbf{Wiener--Pitt phenomenon}. The first rigorous proof of the existence of the Wiener--Pitt phenomenon was given by Y. Schreider \cite{schreider} and simplified later by Williamson \cite{wil}. An alternative approach (using Riesz products) was introduced by C.~C. Graham \cite{graham} and nowadays we know many examples of measures with a non-natural spectrum (larger than the closure of the set of values of the Fourier--Stieltjes transform) -- consult for example Chapters 5--7 in \cite{grmc}. The set of all measures with a natural spectrum does not admit any convenient algebraic structure, in particular it is not closed under addition (that was first proved by M. Zafran in \cite{Zafran} for $I$-groups, the general case was settled by O. Hatori and E. Sato in \cite{h} and \cite{hs}, for more information on the set of measures with a natural spectrum consult \cite{ow2}). Nonetheless, thanks to the pioneering work of M. Zafran \cite{Zafran} we know that the situation changes drastically when we restrict our consideration to measures with Fourier--Stieltjes transforms vanishing at infinity; in fact, measures with a natural spectrum from this class form a closed ideal. There are many other results connected to this topic but let us limit the discussion to the remarkable paper of F. Parreau \cite{p} containing the construction of a measure with a real but non-natural spectrum (the solution of the so-called Katznelson's problem) and to probably the most up-to-date monograph in this area by B. Host, J.-F. M\'{e}la, F. Parreau \cite{hmp}. Note also that similar problems are considered in a more abstract setting of convolution measure algebras in the book of J. L. Taylor \cite{t} and the quantitative version of the Wiener--Pitt phenomenon is discussed in the paper of N. Nikolski \cite{nik}.

The amount of complicated problems connected with the spectral properties of measures suggests that the structure of the Gelfand space $\triangle(\mathrm{M}(G))$ (for a locally compact Abelian group $G$) is very sophisticated. Indeed, the Wiener--Pitt phenomenon itself implies that the dual group $\widehat{G}$ (identified with the evaluations of the Fourier--Stieltjes transform) is not dense in $\triangle(\mathrm{M}(G))$ but the general question about the separability of $\triangle(\mathrm{M}(G))$ was answered negatively only very recently (see \cite{owg}). For completeness, let us recall that there is a description of $\triangle(\mathrm{M}(G))$ in terms of \textbf{generalised characters} introduced by Y. Schreider in \cite{schreider} and extensively used by various authors but not effective enough to provide a convenient way to calculate spectra of measures.

Since $\mathrm{B}(G)$ is a natural generalisation of the algebra of measures to the non-commutative setting, it is legitimate and tempting to investigate problems analogous to the ones discussed above. But in contrast with the Abelian case, there are only a few papers in the literature dealing with them -- the work of M. E. Walter \cite{wa} in the spirit of Taylor's book \cite{t} and some general remarks given by E. Kaniuth, A. T. Lau and A. \"{U}lger in \cite{klu}. The aim of this paper is to present an account, as complete as possible, of the possibility of extending the Abelian theory to the context of non-commutative groups.

We would like to stress that the intended audience of this paper is the community of classical harmonic analysts. Therefore we decided that, whenever feasible, we will give many details in the operator algebraic parts of the paper, and we often reprove known results, for the reader's convenience (mostly in Subsection \ref{mssandas}).

Let us finish this subsection with the most important definition appearing in this paper.
\begin{de}
We say that $f\in \mathrm{B}(G)$ has a \textbf{natural spectrum} if $\sigma(f)=\overline{f(G)}$. The set of all such elements will be denoted by $\mathrm{NS}(G)$.
\end{de}

\subsection{Contents page}
In this part we provide a sketch of the content of subsequent sections for the convenience of the reader.

Section \ref{basic results} contains the basic results on elements with a natural spectrum (in particular, the inclusions $\mathrm{A}(G), \mathrm{B}(G)\cap \mathrm{AP}(G)\subset \mathrm{NS}(G)$ are proved in Fact \ref{zawpod}). Additionally, the procedure of extending a positive definite function from a subgroup is discussed in the context of spectral properties (Proposition \ref{prz}), which enables us to show that the Wiener-Pitt phenomenon occurs for an arbitrary discrete group containing an infinite Abelian subgroup (Theorem \ref{wpf}) and also in Theorem \ref{nieos} the existence of an uncountable family of pairwise disjoint open sets in $\triangle(\mathrm{B}_{0}(G))$ is justified basing on the results from \cite{owg}.

Section \ref{decomposition} deals with the topic of the `algebraic size' and algebraic properties of the set of all elements with a natural spectrum. The extension of the theorem of Hatori and Sato \cite{hs} on the decomposition of an arbitrary measure on a compact Abelian group into a sum of two measures with a natural spectrum and a discrete measure is proved for maximally almost periodic (discrete) groups (Theorem \ref{rozklad}). The section is concluded with some remarks on the set of spectrally reasonable elements (perturbing each member of $\mathrm{NS}(G)$ to a member of $\mathrm{NS}(G)$), which was the main object under study (for $G=\mathbb{Z}$) in \cite{ow2}.

The purpose of Section \ref{Section Zafran} is to transfer the results from the very important Zafran's paper \cite{Zafran} on the set of measures in $\mathrm{M}_{0}$ with a natural spectrum to the present setting. The main Theorem \ref{glz} says that $\mathrm{B}_{0}(G)\cap \mathrm{NS}(G)$ (where $\mathrm{B}_{0}(G):= \mathrm{B}(G)\cap c_{0}(G)$) is a closed ideal in $\mathrm{B}(G)$ and the Gelfand transform of each $f\in \mathrm{B}_{0}(G)\cap \mathrm{NS}(G)$ vanishes off $G$. In order to go further, we discuss in detail in Subsection \ref{mssandas} the notions of absolute continuity and mutual singularity of elements of $\mathrm{B}(G)$ in comparison to the classical (for $G$ -- Abelian) situation (in Proposition \ref{rowno}) and the equivalent conditions for a subspace of $\mathrm{B}(G)$ to be orthogonally complemented (Proposition \ref{Prop:Lspaces}). After this preparation we provide a criterion for the non-naturality of the spectrum of an element of $\mathrm{B}_{0}(G)$ (Theorem \ref{nz}) and a more general result on singularity with respect to $\mathrm{B}_{0}(G)\cap \mathrm{NS}(G)$ in Proposition \ref{singzaf}. Along the way we prove a nice formula for the support of the product of two positive elements of $\mathrm{B}(G)$ (Proposition \ref{wzornos}).

Section \ref{approach1} is devoted to establishing the non-commutative analogues of generalised characters. As it was mentioned before, the description of $\triangle(\mathrm{B}(G))$ in these terms is not sufficient to obtain an algorithm for calculating spectra, but is adequate for giving a sufficient condition for a member of $\triangle(\mathrm{B}(G))$ to belong to the Shilov boundary of $B(G)$ (Theorem \ref{brzsz}).

In Section \ref{free groups} we discuss some examples of positive definite functions on free groups, paying special attention to their spectral properties. We start with the famous \textbf{Haagerup function} and then we pass to more sophisticated construction of \textbf{free Riesz products}.

The paper is summarised with a list of remarks and open problems.
\subsection{Definitions and Notation}
As we are going to use a lot of Banach algebra theory (check \cite{kan} or \cite{z} for a complete exposition of this topic), let us recall two important notions associated with an element $x$ of a Banach algebra $A$ -- the spectrum and the spectral radius:
\begin{align*}
\sigma(x)&=\{\lambda\in\mathbb{C}:x-\lambda e \text{ is not invertible}\} \text{ where $e$ is a unit element of $A$}, \\
r(x)&=\sup\{|\lambda|:\lambda\in\sigma(x)\}=\lim_{n\rightarrow\infty}\|x^{n}\|^{\frac{1}{n}}\leq \|x\|.
\end{align*}
In case of a non-unital Banach algebra $A$ we consider the spectrum of $x\in A$ in the unitisation of $A$. The irreplaceable tool in the investigations of spectral properties of elements in a commutative Banach algebra is the set of multiplicative-linear functionals, which we denote by $\triangle(A)$. This set equipped with the weak$^{\ast}$  topology becomes a locally compact Hausdorff space (compact when $A$ is unital) called the Gelfand space of the (commutative) Banach algebra $A$. We define the Gelfand transform of an element $x\in A$ as a continuous function $\widehat{x}:\triangle(A)\rightarrow\mathbb{C}$ via the formula:
\begin{equation*}
\widehat{x}(\varphi):=\varphi(x)\text{ for $\varphi\in\triangle(A)$}.
\end{equation*}
For a commutative unital Banach algebra $A$ the spectrum of an element $x\in A$ coincides with the image of its Gelfand transform. There is a distinguished compact subset of $\triangle(A)$, called the Shilov boundary of $A$ and denoted $\partial(A)$, with the property that every Gelfand transform attains the maximum of its modulus on $\partial(A)$ and being a minimal closed subset of $\triangle(A)$ with this property.

We will discuss now the main Banach algebras under study in the whole paper. Let $G$ be a discrete group and let $\mathrm{B}(G)$ denote the Fourier--Stieltjes algebra of the group $G$, i.e. the linear span of all positive definite functions on $G$ or equivalently the set of all coefficients of unitary representations -- functions of the form $G\ni x\mapsto \langle \xi,\pi(x)\eta\rangle$, where $\pi:G\rightarrow \mathrm{ B}(\mathsf{H}_{\pi})$ is a continuous unitary representation on a Hilbert space $\mathsf{H}_{\pi}$ and $\xi,\eta\in \mathsf{H}_{\pi}$ (for basic facts on $\mathrm{B}(G)$ and related matters consult \cite{ey}). Recall that $\mathrm{B}(G)$ is the dual space of the (full) group $C^{\ast}$-algebra (abbreviated $C^{\ast}(G)$) with the pairing defined as
\begin{equation*}
\langle f,x\rangle:=\sum_{g\in G}f(g)x(g)\text{ for $f\in \mathrm{B}(G)$ and $x\in \mathbb{C}[G]$}.
\end{equation*}
Moreover, $\mathrm{B}(G)$ is a commutative and semisimple unital Banach algebra when considered with the pointwise multiplication. Note that in the classical case ($G$ -- Abelian) we have a natural identification of $\mathrm{B}(G)$ with $\mathrm{M}(\widehat{G})$, given by Bochner's theorem, where $\widehat{G}$ is the dual group of $G$ and $\mathrm{M}(\widehat{G})$ is the algebra of complex Borel measures on $\widehat{G}$.
It is not surprising (as $\mathrm{B}(G)$ is the dual of $C^{\ast}(G)$) that we will also use C$^{\ast}$-algebra theory and the most important features for us will be reminded when they become relevant (see subsection \ref{mssandas}).

There is an important closed ideal in $\mathrm{B}(G)$ called the Fourier algebra of $G$ and denoted $\mathrm{A}(G)$. It is defined as the set of all coefficients of the left regular representation or equivalently as $L^{2}(G)\ast L^{2}(G)$ and serves as a substitute of the classical $L^{1}$ since for an Abelian group $G$ we have a canonical identification $\mathrm{A}(G)\simeq L^{1}(\widehat{G})$. We will also make use of \textbf{almost periodic functions} defined as bounded functions on $G$ such that the set of all their translates is precompact in the uniform topology. It turns out that the class of almost periodic functions is an algebra denoted $\mathrm{AP}(G)$ and moreover $\mathrm{B}(G)\cap \mathrm{AP}(G)$ is a closed subalgebra of $\mathrm{B}(G)$ with the Gelfand space identified with $bG$ -- the Bohr compactification of $G$. The algebra $\mathrm{B}(G)\cap \mathrm{AP}(G)$ is an analogue of the algebra of discrete measures (purely atomic).

Let us move on now to the commutative case (the standard reference for this part is \cite{r}). For a locally compact Abelian group $G$ we denote by $\mathrm{M}(G)$ the Banach algebra of complex Borel measures on $G$ and identify $L^{1}(G)$ with absolutely continuous measures via Radon--Nikodym theorem. $\mathrm{M}_{d}(G)$ will stand for the subalgebra of $\mathrm{M}(G)$ consisting of all discrete measures and $\mathrm{M}_{c}(G)$ will stand for the closed ideal of all continuous measures (vanishing on points). We define the Fourier--Stieltjes transform of a measure $\mu\in \mathrm{M}(G)$ by the formula
\begin{equation*}
\widehat{\mu}(\gamma)=\int_{G}\gamma(-x)d\mu(x)\text{, $\gamma\in\widehat{G}$}.
\end{equation*}
The closed ideal of measures with Fourier--Stieltjes transforms vanishing at infinity will be denoted by $\mathrm{M}_{0}(G)$. It is well-known that $L^{1}(G)\subset \mathrm{M}_{0}(G)\subset \mathrm{M}_{c}(G)$.
\section{Basic results}\label{basic results}
In this section we collect the basic results concerning spectra of elements of $\mathrm{B}(G)$ and properties of its Gelfand space in case when $G$ is a discrete group containing an infinite Abelian subgroup. We start with some examples of elements in $\mathrm{B}(G)$ with a natural spectrum (for a more complete result in this direction consult Remark \ref{spine} at the end of this section).
\begin{f}\label{zawpod}
Let $G$ be a locally compact group. Then $\mathrm{A}(G)\subset \mathrm{NS}(G)$ and $\mathrm{B}(G)\cap \mathrm{AP}(G)\subset \mathrm{NS}(G)$. Moreover,
\begin{equation*}
r(\mathrm{A}(G)):=\{f\in \mathrm{B}(G):\exists_{n\in\mathbb{N}}f^{n}\in \mathrm{A}(G)\}\subset \mathrm{NS}(G).
\end{equation*}
\end{f}
\begin{proof}
The first assertion follows from the fact that $\mathrm{A}(G)$ is an ideal in $\mathrm{B}(G)$. Indeed, $\triangle(\mathrm{A}(G))=G$ and $\triangle(\mathrm{B}(G))=\triangle(\mathrm{A}(G))\cup h(\mathrm{A}(G))$, where $h(\mathrm{A}(G))=\{\varphi\in\triangle(\mathrm{B}(G)):\varphi|_{\mathrm{A}(G)}=0\}$ (for a proof check Lemma 2.2.15 in \cite{kan}). So, for $f\in \mathrm{A}(G)$ we have $\overline{f(G)}=f(G)\cup\{0\}=\sigma_{\mathrm{A}(G)}(f)=\sigma_{\mathrm{B}(G)}(f)$ (the spectrum in $\mathrm{A}(G)$ is calculated in its unitisation). If there exists $n\in\mathbb{N}$ such that $f^{n}\in \mathrm{A}(G)$ then for each $\varphi\in h(\mathrm{A}(G))$ we have $0=\varphi(f^{n})=(\varphi(f))^{n}$ which implies $\varphi(f)=0$. Since the spectrum of an element is the image of its Gelfand transform, we obtain $\sigma_{\mathrm{B}(G)}(f)=\overline{f(G)}=f(G)\cup\{0\}$.

For $\mathrm{B}(G)\cap \mathrm{AP}(G)$ we argue as follows: $\triangle(\mathrm{B}(G)\cap \mathrm{AP}(G))=bG$ ($bG$ is the Bohr compactification of $G$). This implies $\sigma_{\mathrm{B}(G)\cap \mathrm{AP}(G)}(f)=\overline{f(G)}$. But the spectrum in a closed subalgebra cannot decrease and clearly $\overline{f(G)}\subset \sigma_{\mathrm{B}(G)}(f)$ which finishes the proof.
\end{proof}
It is well-known that if $H$ is a subgroup of a discrete group $G$ then every element of $\mathrm{B}(H)$ extends to an element of $\mathrm{B}(G)$ by putting the value zero outside of $H$ (that was first proved on the real line in \cite{he}, for the general case check 32.43 in \cite{hr}). For $f\in \mathrm{B}(H)$ let us denote this extension by $\widetilde{f}\in \mathrm{B}(G)$. On the other hand, $\mathrm{B}(H)$ can be regarded as a closed ideal of $\mathrm{B}(G)$ (consisting of all elements of $\mathrm{B}(G)$ vanishing out of $H$) and when we write $\sigma_{\mathrm{B}(H)}(g)$ for some element of $\mathrm{B}(G)$ supported on $H$ then we mean the spectrum of $g$ in the unitisation of $\mathrm{B}(H)$ treated as a an ideal of $\mathrm{B}(G)$. The next proposition will be crucial in order to extend results from the commutative setting.
\begin{prop}\label{prz}
Let $H$ be a subgroup of the discrete group $G$. Then, for every $f\in \mathrm{B}(H)$ we have $\sigma_{\mathrm{B}(H)}(f)\cup\{0\}=\sigma_{\mathrm{B}(G)}(\widetilde{f})=\sigma_{\mathrm{B}(H)}(\widetilde{f})$.
\end{prop}
\begin{proof}
The last equality holds true because $\mathrm{B}(H)$ is an ideal in $\mathrm{B}(G)$ (see the explanation in the proof of the last Fact). Clearly $0\in\sigma_{\mathrm{B}(G)}(\widetilde{f})$ so let us take $0\neq\lambda\notin\sigma_{\mathrm{B}(H)}(f)$. Then there exists $g\in \mathrm{B}(H)$ such that $(f-\lambda\mathbf{1}_{H})g=\mathbf{1}_{H}$. It is easy to check that $(\widetilde{f}-\lambda\mathbf{1}_{G})(\widetilde{g}-\frac{1}{\lambda}\mathbf{1}_{G\setminus H})=\mathbf{1}_{G}$ which means that $\lambda\notin\sigma_{\mathrm{B}(G)}(\widetilde{f})$. The reverse inclusion is even simpler: the restriction to $H$ of the inverse in $\mathrm{B}(G)$ is an inverse in $\mathrm{B}(H)$.
\end{proof}
We are ready now to prove that the Wiener--Pitt phenomenon occurs for all discrete groups containing infinite Abelian subgroups.
\begin{thm}\label{wpf}
Let $G$ be a discrete group containing an infinite Abelian subgroup. Then there exists $f\in \mathrm{B}_{0}(G)$ with a non-natural spectrum. In particular, $G$ is not dense in $\triangle(\mathrm{B}_{0}(G))$.
\end{thm}
\begin{proof}
Let $H$ be an infinite Abelian subgroup of $G$. Then there exists $g\in \mathrm{B}(H)$ such that $\sigma_{\mathrm{B}(H)}(g)\neq\overline{g(H)}\cup\{0\}$ (take for $g$ a Fourier--Stieltjes transform of some special measure, for example a Riesz product). Now we have $\sigma_{\mathrm{B}(G)}(\widetilde{g})=\sigma_{\mathrm{B}(H)}(g)\cup\{0\}$ by the last proposition. But $\overline{\widehat{g}(G)}=\overline{g(H)}\cup\{0\}$, which completes the proof.
\end{proof}
We are able to prove much more basing on the results from \cite{owg}.
\begin{thm}\label{nieos}
Let $G$ be a discrete Abelian group with an infinite Abelian subgroup. Then $\triangle(\mathrm{B}_{0}(G))$ contains continuum many pairwise disjoint open sets. In particular, this space is not separable. The same conclusion holds true for $\triangle(\mathrm{B}(G))$.
\end{thm}
\begin{proof}
The main idea behind the proof of this theorem in the Abelian case is to construct continuum many measures with Fourier--Stieltjes transforms vanishing at infinity such that each of these measures does not have a natural spectrum and the convolution of two distinct ones is absolutely continuous. So, let us take such a family of measures on $\widehat{H}$ where $H$ is an infinite Abelian subgroup of $G$ and let us denote their Fourier--Stieltjes transforms by $\{f_{\alpha}\}_{\alpha\in\mathbb{R}}\subset B_{0}(H)$; clearly, by Proposition \ref{prz}, the family of extensions $\{\widetilde{f_{\alpha}}\}_{\alpha\in \mathrm{B}_{0}(G)}$ has the same properties. Now we can repeat the argument given in the proof of Theorem 8 in \cite{owg} or argue as follows: we use directly the assertion of Theorem 8 in \cite{owg} to obtain the desired collection of open sets in $\triangle(\mathrm{B}_{0}(H))\subset\triangle(\mathrm{B}(H))$. But recalling that $\mathrm{B}(H)$ and $\mathrm{B}_{0}(H)$ are closed ideals in $\mathrm{B}(G)$ we are able to identify $\triangle(\mathrm{B}(H))$ and $\triangle(\mathrm{B}_{0}(H))$ with open subsets in $\triangle(\mathrm{B}(G))$ or $\triangle(\mathrm{B}_{0}(G))$ (respectively) and the proof is finished.
\end{proof}
\begin{rem}\label{spine}
It is possible to describe a much broader class of examples of elements with a natural spectrum in $B(G)$ than in Fact \ref{zawpod}, basing on the notion of a \textbf{spine} of a Fourier--Stieltjes algebra (abbreviated $A^{\ast}(G)$), which was introduced in \cite{is1} (see also \cite{is2}) and is a non-commutative generalisation of the subalgebra of a measure algebra spanned by maximal group subalgebras (algebras of the form $L^{1}(G_{\tau})$ where $\tau$ is a topology on $G$ finer then the original one and such that $G_{\tau}$ is a locally compact Abelian group). Since the referred definition is quite complicated we will not recall the whole construction, restricting the discussion to the properties which are essential for proving the inclusion $A^{\ast}(G)\subset \mathcal{NS}(G)$.

First of all, we obviously have $\sigma_{B(G)}\subset\sigma_{A^{\ast}(G)}(f)$ for every $f\in A^{\ast}(G)$ so it suffices to show that $\sigma_{A^{\ast}(G)}(f)=\overline{f(G)}$. Our strategy is to prove that $G\subset G^{\ast}:=\triangle(A^{\ast}(G))$ is a dense subset. It is plausible that one can achieve this goal by a careful study of the description of $G^{\ast}$ (see Theorem 4.1 in \cite{is1}) but we will present a more abstract approach. By the definition of $A^{\ast}(G)$ given on the second page of \cite{is1} (see also (3.2) and Theorem 3.5 therein) it is immediate that $A^{\ast}(G)$ is spanned by a certain family of subalgebras, each of them isomorphic to a Fourier algebra on some locally compact group (check Subsection 3.1 in the aforementioned paper). However, as it is well-known (see for example \cite{ey}), the Fourier algebras are commutative regular\footnote{A commutative Banach algebra $A$ is called \textbf{regular} if for every compact subset $E\subset \triangle(A)$ and $\varphi_{0}\notin E$ there exists $x\in A$ such that $\widehat{x}|_{E}=0$ but $\widehat{x}(\varphi_{0})\neq 0$.} Banach algebras, so $A^{\ast}(G)$ is a commutative Banach algebra spanned by regular subalgebras and by the theorem on the existence of the largest regular subalgebra (check Theorem 4.3.2 in \cite{kan}) $A^{\ast}(G)$ is regular. This immediately implies $\overline{G}=G^{\ast}$ (suppose that this is not the case and take $E=\overline{G}$ in the definition of regularity).

Note that the inclusion $A^{\ast}(G)\subset\mathcal{NS}(G)$ does not follow automatically from the fact that $A^{\ast}(G)$ is spanned by elements with a natural spectrum since the sum of two elements with a natural spectrum does not necessarily have a natural spectrum (this topic is elaborated on in the next section); even in the commutative setting we do not have an elementary argument at our disposal (excluding some simple cases like $G=\mathbb{T}$ or $G=\mathbb{R}$).
\end{rem}

\section{Decomposition of \texorpdfstring{$\mathrm{B}(G)$}{B(G)}}\label{decomposition}
In this section we will show that the set of all elements from $\mathrm{B}(G)$ with a natural spectrum is `big', at least in the algebraic sense. The main theorem is an extension of the result of O. Hatori and E. Sato (see \cite{hs} and \cite{o}, for the non-compact case consult \cite{h}) for compact Abelian groups, which originally states that every measure on a compact Abelian group is a sum of two measures of a natural spectrum and a discrete measure (also with a natural spectrum). The obstacle in the direct transfer of the classical proof lies in the fact that (contrary to the commutative case) not all discrete groups embed into their Bohr compactifications. Therefore we restrict ourselves to \textbf{maximally almost periodic groups} -- the class of groups for which the natural mapping into their Bohr compactification is injective. This class, characterised by the residual finiteness (for finitely generated groups) or by the possibility of embedding into a compact group, contains many interesting non-commutative examples such as free groups or, more generally, many hyperbolic groups (conjecturally, all of them). We conclude this section with some comments on \textbf{spectrally reasonable elements} -- the elements from $\mathrm{B}(G)$ which perturb any element with a natural spectrum to an element with a natural spectrum.
\subsection{Theorem of Hatori and Sato}
In the proof of the analogue of the theorem of Hatori and Sato we will repeatedly use the regularity of $A(G)$ (for a proof for compact $G$ see Lemma 2.9.5 in \cite{kan}, the general case is treated in \cite{ey}) and 'normality' of regular Banach algebras which is formalized in the following lemma (check Corollary 4.2.9 in \cite{kan}).
\begin{lem}\label{nor}
Every regular commutative Banach algebra $A$ is normal in the sense that whenever $E\subset\triangle(A)$ is closed, $K\subset\triangle(A)$ is compact and $E\cap K=\emptyset$, then there exists $x\in A$ such that $\mathrm{supp}\widehat{x}\subset\triangle(A)\setminus E$ and $\widehat{x}|_{K}=1$.
\end{lem}

\begin{thm}\label{rozklad}
Let $G$ be a discrete infinite maximally almost periodic group. Then $\mathrm{B}(G)=\mathrm{NS}(G)+\mathrm{NS}(G)+\mathrm{B}(G)\cap \mathrm{AP}(G)$.
\end{thm}
\begin{proof}
We will denote by $b: G \hookrightarrow bG$ the embedding of $G$ into its Bohr compactification $bG$. By the theorem of Zelmanov (\cite{zel}) $bG$, being a compact group, contains an infinite Abelian subgroup, whose closure will be denoted by $H$; then $H$ is a compact commutative group. Let us take open sets $U_{0},U_{1}\subset H$ such that $\overline{U_{0}}\cap\overline{U_{1}}=\emptyset$.
By Theorems 41.5 and 41.13 in \cite{hr} there exist Helson sets\footnote{A compact subset $K \subset G$ of a locally compact group $G$ is called a \textbf{Helson set} if any continuous function on $K$ can be extended to an element of $\mathrm{A}(G)$.} $K_{0}\subset U_{0}$ and $K_{1}\subset U_{1}$ homeomorphic to the middle-third Cantor  set $C$. Using Alexandroff--Hausdorff theorem (cf. \cite[Theorem 3-28]{HY}), which states that any compact Hausdorff space is a continuous image of the Cantor set, we get surjective functions $p_i\colon K_{i} \to \overline{\mathbb{D}}$ for $i=1,2$. Recalling the definition of a Helson set and using Lemma \ref{nor} we obtain the existence of functions $s_{0},s_{1}\in \mathrm{A}(H)$ such that $s_{i}(K_{i})=\overline{\mathbb{D}}$ and $\mathrm{supp}(s_{i})\subset U_{i}$. We extend the functions $s_{0},s_{1}\in \mathrm{A}(H)$ to elements of $\widetilde{s_{0}}, \widetilde{s_{1}}\in A(bG)$ (the existence of such an extension is proved in Theorem 6.4 from \cite{dr}). As a compact Hausdorff space is automatically normal we can find two disjoint open sets $V_{0},V_{1}\subset bG$ such that $\overline{U_{0}}\subset V_{0}$ and $\overline{U_{1}}\subset V_{1}$. Applying Lemma \ref{nor} we are allowed to find $\widetilde{g_{0}},\widetilde{g_{1}}\in A(bG)$ satisfying $\widetilde{g_{i}}|_{\overline{U_{i}}}=\widetilde{s_{i}}|_{\overline{U_{i}}}=s_{i}|_{\overline{U_{i}}}$ and $\mathrm{supp}\widetilde{g_{i}}\subset V_{i}$. It follows that $\mathrm{supp}\widetilde{g_{0}}\cap \mathrm{supp}\widetilde{g_{1}}=\emptyset$ thus (by Lemma \ref{nor} again) there exists $\widetilde{h}\in A(bG)$ with
\begin{equation*}
\widetilde{h}(t)=\left\{  \begin{array}{c}
                0\text{ for } t\in\mathrm{supp}\widetilde{g_{0}} \\
                1\text{ for } t\in\mathrm{supp}\widetilde{g_{1}}\\
              \end{array}\right.
\end{equation*}

Let $f\in \mathrm{B}(G)$ and put $f_{0}:=f\cdot h$ and $f_{1}:= f- f_{0}$ where $h$ is the restriction of $\widetilde{h}$ to $G$. Denote by $r_{0}$, $r_{1}$ the spectral radii of $f_{0}$, $f_{1}$ (respectively). We also define $l_{0}:= f_{0}+r_{0}g_{0}$, $l_{1}:=f_{1}+r_{1}g_{1}$, and $l_{2}:=-r_{0}g_{0}-r_{1}g_{1}\in \mathrm{B}(G)\cap \mathrm{AP}(G)$. Then $f=l_{0}+l_{1}+l_{2}$ and it is enough to prove that $l_{0}\in \mathrm{NS}(G)$ (the statement about $l_{1}$ is justified in the same way).

By the definitions of $h$ and $g_{0}$ we have $h\cdot g_{0}=0$ which implies $f_{0}\cdot g_{0}=0$. This gives
\begin{equation}\label{rown}
l_{0}(G)\cup\{0\}=f_{0}(G)\cup r_{0}g_{0}(G).
\end{equation}
But $\overline{g_{0}(G)}=\widetilde{g_{0}}(bG)\supset\overline{\mathbb{D}}$, so by the definition of the spectral radius we have from $(\ref{rown})$,
\begin{equation}\label{zaw1}
\overline{l_{0}(G)}\subset\overline{r_{0}g_{0}(G)}.
\end{equation}
Now, if $g_{0}(x)\neq 0$ for some $x\in G$ then $f_{0}(x)=0$, so $l_{0}(x)=r_{0}g_{0}(x)$, which gives $r_{0}g_{0}(G)\setminus\{0\}\subset \overline{l_{0}(G)}$. Together with $(\ref{zaw1})$, taking into account the closedness of $\overline{l_{0}(G)}$, we obtain
\begin{equation}\label{toz1}
\overline{l_{0}(G)}=\overline{r_{0}g_{0}(G)}.
\end{equation}
Using again $f_{0}\cdot g_{0}=0$ we obtain $\sigma(l_{0})\subset \sigma(f_{0})\cup\sigma(r_{0}g_{0})$. The naturality of the spectrum of $g_{0}$, the fact $\overline{r_{0}g_{0}(G)}\supset r_{0}\overline{\mathbb{D}}$ and the definition of $r_{0}$ leads to
\begin{equation*}
\sigma(l_{0})\subset\sigma(r_{0}g_{0})=\overline{r_{0}g_{0}(G)}=\overline{l_{0}(G)}\text{ by $(\ref{toz1})$.}
\end{equation*}
That finishes the proof, since the inclusion $\overline{l_{0}(G)}\subset\sigma(l_{0})$ is trivial.
\end{proof}
In general we cannot rely on almost periodic functions since there may be very few of them (there are $G$ groups called \textbf{minimally almost periodic} such that $\mathrm{AP}(G)$ is trivial, an example of such a group is $\mathrm{SL}(2,\mathbb{R})$, check \cite{hk}). But we can prove a similar theorem replacing the summand $\mathrm{AP}(G)\cap \mathrm{NS}(G)$ by another $\mathrm{NS}(G)$. The proof is a straightforward adaption of the original idea from \cite{hs} and application of the extension procedure described in the previous section (note that as we extend our functions by zero, the algebraic properties, such as the product of two elements being equal to zero, are preserved).
\begin{thm}\label{dodd}
Let $G$ be a discrete group containing an infinite Abelian subgroup. Then $\mathrm{B}(G)=\mathrm{NS}(G)+\mathrm{NS}(G)+\mathrm{NS}(G)$.
\end{thm}
\subsection{Spectrally reasonable elements}
It follows from Theorem \ref{dodd} that the set $\mathrm{NS}(G)$ is not closed under addition for discrete groups $G$ containing an infinite Abelian subgroups so it is meaningful to introduce the notion of elements which perturb the elements with a natural spectrum to elements with a natural spectrum. This concept was studied extensively in \cite{ow2} (for Abelian groups) and all facts stated in this subsection can be proved analogously in the present context as in the cited paper and consequently the arguments are omitted.
\begin{de}
We say that $f\in \mathrm{B}(G)$ is \textbf{spectrally reasonable}, if $f+g\in \mathrm{NS}(G)$ for all $g \in \mathrm{NS}(G)$. The set of all spectrally reasonable measures will be denoted by $\mathrm{S}(G)$.
\end{de}
Let us collect first the basic information on the set $\mathrm{NS}(G)$ itself.
\begin{prop}\label{mon}
The set of all elements with natural spectra is a self-adjoint ($f\in \mathrm{NS}(G)$ $\Rightarrow$ $\overline{f}\in \mathrm{NS}(G)$), closed subset of $\mathrm{B}(G)$ which is also closed under multiplication by complex numbers. Moreover, this set is closed under functional calculus, i.e. for $f\in \mathrm{B}(G)$ and a holomorphic function $F$ defined on some open neighbourhood of $\sigma(f)$ we have $F(f)\in \mathrm{NS}(G)$.
\end{prop}
Unlike $\mathrm{NS}(G)$, the set $\mathrm{S}(G)$ has a rich algebraic structure.
\begin{thm}\label{pod}
The set $\mathrm{S}(G)$ is a closed, unital $^{\ast}$-subalgebra of $\mathrm{B}(G)$.
\end{thm}
The results of the next section allow us to show the following theorem supplying us with plenty of examples of spectrally reasonable elements.
\begin{thm}
$\mathrm{B}_{0}(G)\cap \mathrm{NS}(G)\subset \mathrm{S}(G)$.
\end{thm}

\section{On \texorpdfstring{$\mathrm{B}_{0}(G)\cap \mathrm{NS}(G)$}{intersection of B0(G) with NS(G)}}\label{Section Zafran}
This section is devoted to proving the analogues of the results from the paper of M. Zafran \cite{Zafran} for a discrete group $G$. Let us recall first the main theorem from this publication (see Theorem 3.2 in \cite{Zafran}).
\begin{thm}[Zafran]
Let $G$ be a compact Abelian group. Let
\begin{equation*}
\mathscr{C}=\{\mu\in M_{0}(G):\sigma(\mu)=\widehat{\mu}(\widehat{G})\cup\{0\}\}.
\end{equation*}
\begin{enumerate}[{\normalfont (i)}]
  \item If $h\in\triangle(M_{0}(G))\setminus{\widehat{G}}$ then $h(\mu)=0$ for all $\mu\in\mathscr{C}$.
  \item $\mathscr{C}$ is a closed ideal in $M_{0}(G)$.
  \item $\triangle(\mathscr{C})=\widehat{G}$.
\end{enumerate}
\end{thm}
The missing ingredient to show the analogous theorem (with $M_{0}(G)$ replaced by $\mathrm{B}_{0}(G):=\mathrm{B}(G)\cap c_{0}(G)$ for $G$-discrete) is the identification of the multiplier algebra of $L^{1}(G)$ with $\mathrm{M}(G)$ (sometimes called Wendel's theorem). In our setting we have to answer the following question: is the multiplier algebra of $\mathrm{A}(G)$ equal to $\mathrm{B}(G)$? This problem was solved by V. Losert in \cite{l} and it turned out that every multiplier on $\mathrm{A}(G)$ is represented by some element of $\mathrm{B}(G)$ if and only if $G$ is amenable.

Taking this into account, we are able to repeat the original Zafran's proof for amenable groups but for non-amenable groups we can only obtain the analogous assertion concerning the whole multiplier algebra of $\mathrm{A}(G)$ (i.e. we replace $\mathrm{B}(G)$ in the above formulation by the multiplier algebra of $\mathrm{A}(G)$) which is not enough since then we calculate the spectrum of an element in a much bigger algebra with non-equivalent norm. In the first subsection we will show how to bypass this difficulty by means of an elementary trick (see Proposition \ref{trr}).

The next part of this section is an elaboration of the properties of the notions of mutual singularity and absolute continuity which will have great importance for the investigations given in the last subsection ($L$-ideal property of Zafran's ideal and a criterion for non-naturality of the spectrum).
\subsection{First results}
\begin{lem}\label{brakiz}
Let $f\in \mathrm{B}(G)$ and let $\lambda$ be an isolated point of $\sigma(f)$. Then there exists $x\in G$ such that $f(x)=\lambda$.
\end{lem}
\begin{proof}
By the assumption we are able to find two disjoint open sets $U,V\subset\mathbb{C}$ such that $\lambda\in U$ and $\sigma(f)\setminus\{\lambda\}\subset V$. Since the sets $U$ and $V$ are disjoint and satisfy $\sigma(f)\subset U\cup V$ the function $H$ defined by the formula
\begin{equation*}
H(t)=\left\{\begin{array}{rl}
           0,&\text{ for }t\in V  \\
           1,&\text{ for }t\in U
         \end{array}\right.
\end{equation*}
is holomorphic on an open set containing $\sigma(f)$. Hence we are allowed to apply the functional calculus obtaining an element $h:= H\circ f\in \mathrm{B}(G)$. Since $\lambda \notin f(G)$, we have $h(x)=0$ for every $x\in G$, yet there exists $\varphi\in\triangle(\mathrm{B}(G))$ such that $\varphi(h)=H(\varphi(f))=H(\lambda)=1$, which results in a contradiction.
\end{proof}
The following proposition shows that if the spectrum of an element is not natural then it is automatically uncountable.
\begin{prop}\label{trr}
Let $f\in \mathrm{B}(G)\setminus \mathrm{NS}(G)$. Then $\sigma(f)$ contains a perfect set. In particular, $\sigma(f)$ has the cardinality of the continuum.
\end{prop}
\begin{proof}
Let $X=\sigma(f)\setminus\overline{f(G)}$. By Lemma \ref{brakiz} the set $X$ is dense-in-itself. Hence $\overline{X}\subset\sigma(f)$ is a perfect set.
\end{proof}
We will use the convention: $\mathrm{B}_{0}(G):= \mathrm{B}(G)\cap c_{0}(G)$.
\begin{lem}\label{przelicz}
Let $f_{1},f_{2}\in \mathrm{B}_{0}(G)\cap \mathrm{NS}(G)$. Then we also have $f_{1}+f_{2}\in \mathrm{B}_{0}(G)\cap \mathrm{NS}(G)$.
\end{lem}
\begin{proof}
By standard Gelfand theory $\sigma(f_{1}+f_{2})\subset\sigma(f_{1})+\sigma(f_{2})$ which finishes the proof by Proposition \ref{trr}, since the algebraic sum of two countable sets is countable.
\end{proof}
We are ready now to prove the analogue of the main theorem from Zafran's paper \cite{Zafran}.
\begin{thm}\label{glz}
The following hold true:
\begin{enumerate}[{\normalfont (i)}]
  \item\label{zafone} If $\varphi\in\triangle(\mathrm{B}(G))\setminus G$ then $\varphi(f)=0$ for every $f\in \mathrm{B}_{0}(G)\cap \mathrm{NS}(G)$.
  \item\label{zaftwo} $\mathrm{B}_{0}(G)\cap \mathrm{NS}(G)$ is a closed ideal in $\mathrm{B}(G)$.
  \item\label{zafthree} $\triangle(\mathrm{B}_{0}(G)\cap \mathrm{NS}(G))=G$.
\end{enumerate}
\end{thm}
\begin{proof}
We start with \eqref{zafone}. Let $f\in \mathrm{B}_{0}(G)\cap \mathrm{NS}(G)$ and let us fix $\varepsilon>0$. Then there exists a finite set $K\subset G$ such that $|f(x)|<\varepsilon$ for $x\notin K$. Let $g\in c_{00}\subset \mathrm{A}(G)$ be equal to $\chi_{K}$. Define $f_{1}:=f\cdot g$ and $f_{2}=f-f_{1}$. Of course, $f=f_{1}+f_{2}$ and $f_{1}\in c_{00}(G)\subset \mathrm{A}(G)$. Since $\mathrm{A}(G)$ is a closed ideal in $\mathrm{B}(G)$ we have the following partition of $\triangle(\mathrm{B}(G))$ (see Lemma 2.2.15 in \cite{kan}):
\begin{equation*}
\triangle(\mathrm{B}(G))=\triangle(\mathrm{A}(G))\cup h(\mathrm{A}(G)),
\end{equation*}
where $h(\mathrm{A}(G))=\{\varphi\in\triangle \mathrm{B}(G):\varphi|_{\mathrm{A}(G)}=0\}$.

Recalling that $\triangle(\mathrm{A}(G))=G$ we get $\varphi(f_{1})=0$ for $\varphi\in\triangle(B(G))\setminus G$. We shall show $|\varphi(f_{2})|<\varepsilon$. From the basics of Gelfand theory we have $\varphi(f_{2})\in\sigma(f_{2})=f_{2}(G)\cup\{0\}$ ($f_{2}\in \mathrm{B}_{0}(G)\cap \mathrm{NS}(G)$ because $f,f_{1}\in \mathrm{B}_{0}(G)\cap \mathrm{NS}(G)$ and $f_{2}=f-f_{1}$ so we are able to use Lemma \ref{przelicz}). If $\varphi(f_{2})=0$ then we are done. Otherwise $\varphi(f_{2})=f_{2}(x)$ for some $x\in G$. For $x\in K$, the definition of $f_{2}$ implies $f_{2}(x)=0$. Finally, if $x\notin K$ then $|\varphi(f_{2})|=|f_{2}(x)|=|f(x)|<\varepsilon$ by the choice of $K$.

Part \eqref{zaftwo} follows immediately from part \eqref{zafone}. Passing to the proof of \eqref{zafthree} suppose that there is $\varphi\in \triangle(\mathrm{B}_{0}(G)\cap \mathrm{NS}(G))\setminus G$ and take any $f\in \mathrm{B}_{0}(G)\cap \mathrm{NS}(G)$. Then decompose $f=f_{1}+f_{2}$ as in the proof of \eqref{zafone}. Exactly the same arguments as before show $\varphi(f_{1})=0$ and $|\varphi(f_{2})|<\varepsilon$ finishing the proof since $\varepsilon>0$ was arbitrary.
\end{proof}
\begin{rem}\label{rad}
The first item of Theorem \ref{glz} implies that $\mathrm{B}_{0}(G)\cap \mathrm{NS}(G)$ is a \textit{radical ideal}, i.e. if $f^{n}\in \mathrm{B}_{0}(G)\cap \mathrm{NS}(G)$ for some $n\in\mathbb{N}$ then $f\in \mathrm{B}_{0}(G)\cap \mathrm{NS}(G)$.
\end{rem}
In order to pass further in generalisations of the results from Zafran's paper \cite{Zafran} we need to establish the notion of mutual singularity and absolute continuity of elements in $\mathrm{B}(G)$.

\subsection{Mutual singularity and absolute continuity in \texorpdfstring{$\mathrm{B}(G)$}{\mathrm{B}(G)}}\label{mssandas}

We start with the reminder concerning the polar decomposition of functionals on C$^{\ast}$-algebras (see for example chapter III, section 4 in \cite{tak}).

Let $A$ be C$^{\ast}$-algebra and let $\varphi\in A^{\ast}$. Then $\varphi$ can be seen as an element of the predual $\mathsf{M}_{\ast}$ of $\mathsf{M}=A^{\ast\ast}$ (the universal enveloping von Neumann algebra of $A$). Recall that any element $x\in\mathsf{M}$ acts on $\varphi$ both from left and right via $(x\varphi)(y):=\varphi(yx)$ and $(\varphi x)(y) = \varphi(xy)$; the convention might seem surprising but it gives $\mathsf{M}_{\ast}$ the structure of an $\mathsf{M}$--$\mathsf{M}$-bimodule. In this setting we have the following \textbf{polar decomposition} of $\varphi$:
\begin{equation}\label{pol}
\varphi=v|\varphi|\text{ where }v\in\mathsf{M} \text{ is a partial isometry and }|\varphi|\in A^{\ast}_{+},
\end{equation}
where we denoted by $A^{\ast}_{+}$ the set of positive functionals on $A$. This polar decomposition comes from the polar decomposition of some element of $\mathsf{M}$. Indeed, there always exists $x \in \mathsf{M}$ of norm one, such that $\varphi(x) = \|\varphi\|$. Using the polar decomposition of $x^{\ast}=v|x^{\ast}|$, we find a partial isometry $v$ and define $|\varphi|:=v^{\ast}\varphi$; it is not hard to check that $|\varphi|$ is positive and $\varphi=v|\varphi|$.

In the case of operators we have a unique polar decomposition, provided that we assume something about the initial space of the partial isometry involved in it. To mimic this criterion in the setting of normal functionals, we need to introduce the notion of support. For each $\varphi\in A^{\ast}_{+}$ there exists the smallest projection $e$ in $\mathsf{M}$ such that $\varphi=e\varphi=\varphi e=e\varphi e$. This projection is called the \textbf{support} of $\varphi$ and is denoted by $\mathrm{s}(\varphi)$. Coming back to the polar decomposition, the partial isometry $v\in\mathsf{M}$ and $|\varphi|\in A^{\ast}_{+}$ is uniquely determined by $\varphi$, the equation ($\ref{pol}$) and the property that $v^{\ast}v$ is the support of $|\varphi|$.

In our investigations we will mainly use the notion of the \textbf{central support} of a linear functional $\varphi\in A^{\ast}$ which is the smallest central projection $\mathrm{zs}(\varphi)$ in $\mathsf{M}$ such that $\varphi$ vanishes on $(1-\mathrm{zs}(\varphi))\mathsf{M}$. Note that the definition of central support does not require positivity of $\varphi$. The usual notion of support is unambiguous for hermitian linear functionals but for general ones we should work with left and right supports, which often is cumbersome. Using the polar decomposition, it is not hard to verify that $\mathrm{zs}(\varphi)=\mathrm{zs}(|\varphi|)$.

If we consider a positive functional, there is an alternative point of view on the central support. Let $\varphi \in A^{\ast}_{+}$. We can perform the GNS construction with respect to it, resulting in a cyclic representation $\pi:A \to \mathrm{B}(\mathsf{H})$, where we have a cyclic vector $\Omega \in \mathsf{H}$ such that $\varphi(x) = \langle \Omega, \pi(x) \Omega \rangle$. This representation extends to a normal representation $\overline{\pi}: \mathsf{M} \to \mathrm{B}(\mathsf{H})$, whose image is the von Neumann algebra generated by $\pi(A)$. The kernel of $\overline{\pi}$ is a weak$^{\ast}$-closed two-sided ideal of the von Neumann algebra $\mathsf{M}$. It is well-known that such ideals are of the form $z\mathsf{M}$ for some central projection $z \in \mathsf{M}$. %It comes from the fact that a weak$^{\ast}$-closed ideal is a von Neumann algebra in its own right, hence it is unital; one can then check that this unit has to be a central projection in the bigger algebra.
Therefore the von Neumann algebra generated by $\pi(A)$ may be identified with $\mathsf{M} \slash z\mathsf{M} \simeq (1-z)\mathsf{M}$, and $(1-z)$ is the central support of $\varphi$.

We borrow the notion of absolute continuity and mutual singularity of functionals on a C$^{\ast}$-algebra from R. Exel (see \cite{ex}).
\begin{de}
Let $\varphi,\psi\in A^{\ast}$ where $A$ is a C$^{\ast}$-algebra. We say that $\varphi$ is absolutely continuous with respect to $\psi$ ($\varphi\ll \psi$) if $\mathrm{zs}(\varphi)\leq \mathrm{zs}(\psi)$; $\varphi$ and $\psi$ are mutually singular ($\varphi\bot \psi$) if $\mathrm{zs}(\varphi)\mathrm{zs}(\psi)=0$.
\end{de}
In our situation $A=C^{\ast}(G)$, $\mathsf{M}_{\ast}=A^{\ast}=\mathrm{B}(G)$ and $\mathsf{M}=\mathrm{B}(G)^{\ast}=\left(C^{\ast}(G)\right)^{\ast\ast}=: W^{\ast}(G)$. We shall use the following analogue of the Lebesgue decomposition (see Proposition 7 in \cite{ex}), whose proof we will include for the reader's convenience.
\begin{prop}\label{rl}
Let $\varphi,\psi\in A^{\ast}$ where $A$ is a C$^{\ast}$-algebra. Then $\varphi$ can be uniquely decomposed as a sum $\varphi=\varphi_{1}+\varphi_{2}$ such that $\varphi_{1}\ll \psi$ and $\varphi_{2}\bot \psi$. Moreover, if $\varphi=u|\varphi|$, $\varphi_{1}=u_{1}|\varphi_{1}|$, and $\varphi_{2}=u_{2}|\varphi_{2}|$ are polar decompositions of $\varphi$, $\varphi_{1}$, and $\varphi_{2}$ then we have
\begin{enumerate}[{\normalfont (i)}]
  \item $|\varphi|=|\varphi_{1}|+|\varphi_{2}|$, $u_{1}=u\cdot \mathrm{zs}(\psi)$, $u_{2}=u(1-\mathrm{zs}(\psi))$,
  \item $\|\varphi\|=\|\varphi_{1}\|+\|\varphi_{2}\|$,
  \item $|\varphi|=|\varphi_{1}|+|\varphi_{2}|$ is the unique decomposition of $|\varphi|$ satisfying $|\varphi_{1}|\ll \psi$ and $|\varphi_{2}|\bot \psi$.
\end{enumerate}
\end{prop}
\begin{proof}
We define $\varphi_1:= \mathrm{zs}(\psi)\varphi$ and $\varphi_{2}:=(1-\mathrm{zs}(\psi))\varphi$; they give the desired decomposition. It is obviously unique, because the $\varphi_{2}$ part in the decomposition vanishes on $\mathrm{zs}(\psi)A^{\ast\ast}$, so $\varphi_{1}$ equals $\mathrm{zs}(\psi)\varphi$ on $\mathrm{zs}(\psi)A^{\ast\ast}$ -- it follows that $(\varphi_{1} - \mathrm{zs}(\psi)\varphi)$ vanishes on both $\mathrm{zs}(\psi)A^{\ast\ast}$ and $(1-\mathrm{zs}(\psi))A^{\ast\ast}$, so it vanishes everywhere, hence $\varphi_{1}=\mathrm{zs}(\psi)\varphi$. The rest of the proof is very similar.
\end{proof}
Before we go any further, let us discuss the Abelian case, in which there is no distinction between the support and the central support. Since the supports that we employ here live in an abstract, huge von Neumann algebra $A^{\ast\ast}$, we would like to present now a characterisation of absolute continuity that does not use the bidual and will allow us to prove that our notions of absolute continuity and singularity correspond to the classical ones in the Abelian case. We begin with a lemma.
\begin{lem}
Let $A$ be a $C^{\ast}$-algebra and let $\varphi \in A^{\ast}$. For any element $x \in A^{\ast\ast}$ we have $x\cdot \varphi \ll \varphi$ and $\varphi\cdot x \ll \varphi$.
\end{lem}
\begin{proof}
Since $\mathrm{zs}(\varphi)$ is central, we get $x(1-\mathrm{zs}(\varphi))y=(1-\mathrm{zs}(\varphi))xy$ for any $y\in A^{\ast\ast}$, therefore both $x\cdot \varphi$ and $\varphi\cdot x$ vanish on $(1-\mathrm{ zs}(\varphi))A^{\ast\ast}$.
\end{proof}
\begin{prop}\label{Prop:translation}
Let $A$ be a $C^{\ast}$-algebra and let $\varphi,\psi \in A^{\ast}$. Then $\varphi \ll \psi$ if and only if we can approximate $\varphi$ in norm by linear combinations of functionals of the form $x\cdot \psi\cdot y$, where $x,y \in A$. In fact, $x$ and $y$ can be drawn from any dense subalgebra $\mathcal{A}$. In particular, if $A=C^{\ast}(G)$ for a discrete group $G$, we can take $\mathcal{A}=\mathbb{C}[G]$ to show that $\varphi \ll \psi$ if and only if $\varphi$ is a limit in $\mathrm{B}(G)$ of linear combinations of (two-sided) translates of $\psi$.
\end{prop}
\begin{proof}
From the lemma it follows that limits of translates of a given functional are absolutely continuous with respect to it, so we just have to verify the reverse implication.

First of all, let us prove that we can approximate $|\psi|$ by translates of $\psi$ so that we may assume that $\psi\geqslant 0$. Using the polar decomposition, we get $\psi=u|\psi|$ and $|\psi|=u^{\ast}\psi$ for some partial isometry $u \in A^{\ast\ast}$. Of course $A$ is weak$^{\ast}$-dense in $A^{\ast\ast}$, so $u^{\ast}$ may be viewed as a weak$^{\ast}$-limit of elements of $A$. In fact, any norm-dense subalgebra $\mathcal{A}$ of $A$ is weak$^{\ast}$-dense in $A^{\ast\ast}$, so we may assume that these elements belong to $\mathcal{A}$. There exists a net $(x_i)_{i\in I} \subset \mathcal{A}$ that tends to $u^{\ast}$ in the weak$^{\ast}$-topology. It follows that the functionals $x_{i} \cdot \psi$ converge weakly to $u^{\ast}\psi=|\psi|$. By passing to convex combinations, we may assume that they converge in norm. Therefore it suffices now to show that if $\varphi \ll \psi$ then $\varphi$ is a norm-limit of two-sided translates of $|\psi|$. Using $|\psi|$, we perform a GNS construction, to obtain a triple $(\pi, \mathsf{H}, \Omega)$, where $\pi:A \to \mathrm{B}(\mathsf{H})$ is a $\ast$-homomorphism and $\Omega \in \mathsf{H}$ is a cyclic vector such that $|\psi|(x) = \langle \Omega, \pi(x) \Omega\rangle$ for $x\in A$. Note that for $y,z\in A$ we get $(y^{\ast}|\psi|z)(x) =\langle \pi(y)\Omega, \pi(x) \pi(z)\Omega\rangle$. Denote by $\mathrm{vN}(|\psi|)$ the von Neumann algebra generated by $\pi(A)$; then these functionals extend to normal functionals on $\mathrm{vN}(|\psi|)$ given by $(y^{\ast}|\psi|z)(x) = \langle \pi(y)\Omega, x \pi(z)\Omega\rangle$ for $x\in\mathrm{vN}(|\psi|)$. Since $\varphi \ll |\psi|$, we know that $\varphi \in \mathrm{vN}(|\psi|)_{\ast}$, the predual of $\mathrm{vN}(|\psi|)$. We will show that functionals $y^{\ast}|\psi|z$ are linearly dense in $\mathrm{vN}(|\psi|)_{\ast}$. If they were not, there would be a non-zero functional on $\mathrm{vN}(|\psi|)_{\ast}$ that vanished on all of them. Such a functional is necessarily represented by an element $x \in \mathrm{vN}(|\psi|)$. It follows that $\langle \pi(y)\Omega, x \pi(z)\Omega\rangle=0$ for all $y,z \in A$. Since $\Omega$ is a cyclic vector, it follows that $\langle \eta, x\xi\rangle=0$ for all $\eta,\xi \in \mathsf{H}$, so $x=0$; this finishes the proof.
\end{proof}
We are ready to connect our notion of support with the classical one.
\begin{prop}\label{rowno}
Let $G$ be a compact Abelian group and let $\mu_{1},\mu_{2} \in \mathrm{M}(G)$. Then
\begin{enumerate}[{\normalfont (i)}]
\item\label{abscontone} $\mu_{1}\ll \mu_{2}$ if and only if $\widehat{\mu_1} \ll \widehat{\mu_2}$ in $B(\widehat{G})$;
\item\label{absconttwo} $\mu_1 \bot \mu_2$ if and only if $\widehat{\mu_1} \bot \widehat{\mu_2}$ in $B(\widehat{G})$.
\end{enumerate}
\end{prop}
\begin{proof}
We will just show \eqref{abscontone}; \eqref{absconttwo} will follow from the fact that both frameworks admit unique Lebesgue decompositions. Assume first that $\mu_{1} \ll \mu_{2}$, then by Radon-Nikodym theorem we get $f \in L^{1}(|\mu_2|)$ such that $d\mu_1 = f d\mu_2$.

The function $f$ may be approximated in norm by trigonometric polynomials. Indeed, suppose that the subspace spanned by characters of $G$ is not dense in $L^{1}(|\mu_2|)$. Then, by Hahn-Banach theorem,  there exists $h\in L^{\infty}(|\mu_2|)$ such that $\int_{G} \gamma(g) h(g) d|\mu_2|(g)=0$ for all characters $\gamma \in \widehat{G}$. It follows that the Fourier--Stieltjes transform of the measure $h |\mu_2|$ is zero, so the measure itself is zero, therefore $h = 0$ $|\mu_2|$-almost everywhere, i.e. it is a zero element in $L^{\infty}(|\mu_2|)$.

Since multiplication by a character corresponds to a shift on the Fourier transform side, we get that $\widehat{\mu_1}$ can be approximated by linear combinations of translates of $\widehat{\mu_2}$, therefore $\widehat{\mu_1} \ll \widehat{\mu_2}$ in $B(\widehat{G})$.

If $\widehat{\mu_1} \ll \widehat{\mu_2}$ then by Proposition \ref{Prop:translation} $\widehat{\mu_1}$ can be approximated by linear combinations of (one-sided, because the group is Abelian) translates of $\widehat{\mu_2}$, therefore $\mu_1$ can be approximated by measures of the form $f\mu_2$, where $f$ is a trigonometric polynomial, so $\mu_1 \ll \mu_2$.
\end{proof}
We now immerse ourselves in the non-Abelian world. To attain comfort in working with our notions of support, etc., let us start from discussing very basic properties.
\begin{f}\label{norm}
If $\varphi,\psi\in A^{\ast}$ are mutually singular then $\|\alpha \varphi+\beta \psi\|=|\alpha|\|\varphi\|+|\beta|\|\psi\|$ for all $\alpha,\beta\in\mathbb{C}$.
\end{f}
\begin{proof}
Since $\varphi$ and $\psi$ are mutually singular, we get $\mathrm{zs}(\varphi)\cdot \mathrm{zs}(\psi)=0$. We can decompose $A^{\ast\ast}= \mathrm{zs}(\varphi)A^{\ast\ast} \oplus (1-\mathrm{zs}(\varphi))A^{\ast\ast}$, where $1-\mathrm{zs}(\varphi) \geqslant \mathrm{zs}(\psi)$. Since $\alpha \varphi = \alpha \mathrm{zs}(\varphi) \varphi$ and $\beta \psi = \beta \mathrm{zs}(\psi)\psi=\beta(1-\mathrm{zs}(\varphi))\psi$, they attain their norms on $\mathrm{zs}(\varphi)A^{\ast\ast}$ and $(1-\mathrm{zs}(\varphi))A^{\ast\ast}$, respectively. The tasks of finding the norms are independent, hence we get $\|\alpha \varphi+\beta \psi\|=|\alpha|\|\varphi\|+|\beta|\|\psi\|$.
\end{proof}
Classically there is a norm characterization of singularity -- two measures $\mu_1$ and $\mu_2$ are mutually singular iff $\|\mu_1 \pm \mu_2\| = \|\mu_1 \| + \|\mu_2\|$.
Indeed, assume that $\mu_{1},\mu_{2}\in \mathrm{M}(G)$ ($G$ is now a locally compact Abelian group) satisfy this condition (the reverse implication is clear). Of course, $\mu_{1},\mu_{2}\ll |\mu_{1}|+|\mu_{2}|$ so there exist $f_{1},f_{2}\in L^{1}(|\mu|+|\nu|)$ such that for every Borel set $A\subset G$ we have
\begin{equation*}
\mu_{1}(A)=\int_{A}f_{1}d(|\mu_{1}|+|\mu_{2}|)\text{ and }\mu_{2}(A)=\int_{A}f_{2}d(|\mu_{1}|+|\mu_{2}|).
\end{equation*}
By standard arguments
\begin{equation*}
\|\mu_{1}\pm\mu_{2}\|=\int_{G}|f_{1}\pm f_{2}|d(|\mu_{1}|+|\mu_{2}|).
\end{equation*}
The assumption implies $|f_{1}\pm f_{2}|=|f_{1}|+|f_{2}|$, $|\mu_{1}|+|\mu_{2}|$-almost everywhere forcing $f_{1}$ and $f_{2}$ to have disjoint supports.

Our aim now is to furnish an example showing that this criterion fails in the non-commutative setting.
\begin{prop}
Let $\pi: S_{3} \to \mathrm{M}_{2}$ be the standard, two-dimensional, irreducible representation of the permutation group $S_3$. Let $\varphi_1(g):= \langle e_{1}, \pi(g) e_1\rangle$ and $\varphi_{2}(g):=\langle e_2, \pi(g) e_2\rangle$, where $(e_{1}, e_{2})$ is the standard basis of $\mathbb{C}^2$ on which $S_3$ acts. Then $\|\varphi_{1}\pm \varphi_{2}\|=\|\varphi_{1}\|+ \|\varphi_{2}\|$ but $\mathrm{ zs}(\varphi_1)=\mathrm{zs}(\varphi_2)$.
\end{prop}
\begin{proof}
Using the three irreducible representations of $S_3$, we get a decomposition of the group algebra $\mathbb{C}[S_3] = \mathbb{C} \oplus \mathbb{C} \oplus \mathrm{M}_{2}$. With this identification, the support of $\varphi_1$ is equal to the orthogonal projection onto $\mathrm{span}(e_1)$ and the support of $\varphi_2$ is equal to the orthogonal projection onto $\mathrm{span}(e_2)$; it follows that $\mathrm{s}(\varphi_1) \mathrm{s}(\varphi_2)=0$. The proof of Fact \ref{norm} also applies in this setting, because we have $\varphi_{1} = \mathrm{s}(\varphi_{1}) \varphi_{1} \mathrm{s}(\varphi_{1})$ and $\varphi_{2} = \mathrm{s}(\varphi_2) \varphi_2 {\textrm s}(\varphi_2)$, so we have $\|\varphi_1 \pm \varphi_2\| = \|\varphi_1\|+\|\varphi_2\|$. Since the only non-zero central projection in $\mathrm{M}_2$ is the identity, we get $\mathrm{zs}(\varphi_1)=\mathrm{zs}(\varphi_2)$.
\end{proof}
Since we will be dealing with two distinct notions of support, it would be good to know, what each of them is good for and when they coincide. We will address the first concern in the course of the article, presenting various results, some of which require the central support but some of them valid only for the usual support. Now we would like to recall one special case in which the support and the central support are equal.
\begin{f}
Let $A$ be a $C^{\ast}$-algebra. If a functional $\varphi\in A^{\ast}_{+}$ is \textbf{tracial}, i.e. $\varphi(xy)=\varphi(yx)$ for $x,y\in A$, then $ \mathrm{s}(\varphi)= \mathrm{zs}(\varphi)$.
\end{f}
\begin{proof}
If $u \in A$ is a unitary then $\varphi(u(xu^{\ast})) = \varphi(xu^{\ast}u) = \varphi(x)$, therefore $\varphi$ is unitarily invariant. It follows that $u\mathrm{s}(\varphi)u^{\ast}=\mathrm{s}(\varphi)$, i.e. $u\mathrm{s}(\varphi) = \mathrm{s}(\varphi)u$. Since unitaries span $A$, we get that $x \mathrm{s}(\varphi) = \mathrm{s}(\varphi)x$ for any $x\in A$. From density of $A$ in $A^{\ast\ast}$ we conclude that $\mathrm{s}(\varphi)$ belongs to the center of $A^{\ast\ast}$, so $\mathrm{zs}(\varphi)=\mathrm{s}(\varphi)$.
\end{proof}
We will now state a criterion for equality of regular and central supports and then use it to show that any positive functional can be approximated in norm by ones whose supports are central.
\begin{f}
Let $A$ be a $C^{\ast}$-algebra and let $\varphi \in A^{\ast}_{+}$. Then $\mathrm{s}(\varphi)=\mathrm{ zs}(\varphi)$ if and only the positive functional induced by $\varphi$ on $\mathrm{vN}(\varphi)$, the von Neumann algebra generated by the image of the GNS representation of $A$, is faithful.
\end{f}
\begin{proof}
By definition of $\mathrm{zs}(\varphi)$ we may write $\mathrm{vN}(\varphi)\simeq \mathrm{zs}(\varphi)A^{\ast\ast}$. The support $\mathrm{s}(\varphi)$ is the biggest projection such that the restriction of $\varphi$ to $\mathrm{s}(\varphi)A^{\ast\ast} \mathrm{s}(\varphi)$ is faithful. Therefore $\mathrm{s}(\varphi)=\mathrm{zs}(\varphi)$ if and only if $\varphi$ is faithful on $\mathrm{zs}(\varphi)A^{\ast\ast}\simeq\mathrm{vN}(\varphi)$.
\end{proof}
Prompted by this fact, we will call such positive functionals \textbf{GNS faithful}.
\begin{prop}\label{Prop:density}
Let $\varphi$ be a positive functional on a separable $C^{\ast}$-algebra $A$. Then for any $\varepsilon > 0$ there exists a GNS faithful $\psi \in A^{\ast}_{+}$ such that $\mathrm{zs}(\varphi) = \mathrm{zs}(\psi)$ and $\|\varphi-\psi\|\leqslant \varepsilon$.
\end{prop}
\begin{proof}
Perform a GNS representation with respect to $\varphi$. We get that $\varphi(x) = \langle \Omega, x \Omega\rangle$ for any $x\in \mathrm{vN}(\varphi)$ and $\Omega \in \mathsf{H}$ is a cyclic vector for $\mathrm{vN}(\varphi)$. We will take $\psi:= \varphi+\sum_{i=1}^{\infty} y_i^{\ast} \varphi y_i$, where the elements $y_i \in \mathrm{vN}(\varphi)$ are chosen appropriately. Recall that $(y_i^{\ast}\varphi y_i)(x) = \langle y_i \Omega, x y_i \Omega\rangle$. Therefore if we ensure that the set of vectors $\{y_i\Omega\}_{i=1}^{\infty}$ is total in $\mathsf{H}$, which we can do, we will get a GNS faithful functional $\psi$. Note that $\|\psi-\varphi\| \leqslant \sum_{i=1}^{\infty}\|\varphi\| \cdot \|y_i\|^2$, so by rescaling $y_i$'s we may assume that $\|\psi-\varphi\|\leqslant \varepsilon$. Clearly, we have $\mathrm{zs}(\varphi) = \mathrm{zs}(\psi)$.
\end{proof}
We will be interested in transferring the classical notions, such as the notion of a singular measure, to our setting and it will be convenient to introduce the following definition.
\begin{de}
A closed subspace $X \subset \mathrm{B}(G)$ is called an \textbf{$L$-space} if the conditions $f \in X$ and $g\ll f$ imply $g\in X$. An $L$-space is called an $L$-algebra ($L$-ideal) if it is additionally a subalgebra (an ideal) of $\mathrm{B}(G)$.
\end{de}
We will now present equivalent conditions for a closed subspace of the Fourier--Stieltjes algebra to be an $L$-space (the proof can also be found in \cite{ftp}, Chapter 1, Lemma 2.1).
\begin{prop}\label{Prop:Lspaces}
Let $X \subset \mathrm{B}(G)$ be a closed subspace. The following conditions are equivalent:
\begin{enumerate}[{\normalfont (i)}]
\item\label{Lspaceone} $X$ is an $L$-space;
\item\label{Lspacetwo} $X$ is $G$-bi-invariant, i.e. $\mathbb{C}[G]\cdot X \subset X$ and $X\cdot \mathbb{C}[G] \subset X$;
\item\label{Lspacethree} There exists a central projection $z \in W^{\ast}(G)$ such that $X = z \mathrm{B}(G)$.
\end{enumerate}
\end{prop}
\begin{proof}
Equivalence of \eqref{Lspaceone} and \eqref{Lspacetwo} follows from Proposition \ref{Prop:translation}. Clearly \eqref{Lspacethree} implies both \eqref{Lspaceone} and \eqref{Lspacetwo}, so it suffices to prove that \eqref{Lspacetwo} implies \eqref{Lspacethree} and this is what we will do now. We plan to show that for any $y\in W^{\ast}(G)$ we have $yX\subset X$ and $Xy \subset X$; then, as we shall see, the existence of $z$ will follow from general theory of von Neumann algebras. By assumption, we have this property for $y \in \mathbb{C}[G]$. Any element $y \in W^{\ast}(G)$ can be approximated by elements of the group algebra in weak$^{\ast}$-topology. Let $(y_i)_{i\in I} \subset \mathbb{C}[G]$ be a net such that $\lim_{i\in I} y_i=y$. For any $f\in X$ the net $(y_{i}f)_{i\in I}$ converges to $yf$ in weak topology, so $yf \in X$, because $X$ is weakly closed, as it is norm-closed. The same holds for $fy$, so $X$ is invariant by left and right actions of $W^{\ast}(G)$. The annihilator $X^{\perp} \subset W^{\ast}(G)$ is therefore a two-sided, weak$^{\ast}$-closed ideal, hence it is of the form $(1-z)W^{\ast}(G)$ for some central projection $z \in W^{\ast}(G)$. It now follows from standard duality of Banach spaces that $X= z\mathrm{B}(G)$.
\end{proof}
%\begin{cor}\label{podp}
%Let $f\in A^{\ast}$. Then the set of all $g\in A^{\ast}$ such that $g\bot f$ is a closed linear subspace of $A^{\ast}$.
%\end{cor}
%We are going to prove now that shifting an element $f\in \mathrm{B}(G)$ does not change its absolute value and in particular, its support. For $f\in \mathrm{B}(G)$ and $g\in G$ we define $f_{g}$ by the formula $f_{g}(x)=f(g^{-1}x)$, $x\in G$.
%\begin{lem}\label{przes}
%Let $f\in \mathrm{B}(G)$ and let us fix $x\in G$. Then $|f_{g}|=|f|$. In particular, $e_{f_{g}}=e_{f}$.
%\end{lem}
%\begin{proof}
%Let $f=v|f|$ be the polar decomposition of $f$ and let$\lambda_{g}$ be defined as $\lambda_{g}(x)=g^{-1}x$. We are going to show that $f_{g}=\lambda_{g}v|f|$ is a polar decomposition of $f_{g}$. Indeed, $(\lambda_{g}v)^{\ast}(\lambda_{g}v)=v^{\ast}\lambda_{g}^{\ast}\lambda_{g}v=v^{\ast}v=e_{f}$ and it is easy to verify that $\lambda_{g}v$ is a partial isometry and moreover $f_{g}=\lambda_{g}f$ which gives the desired formula.
%\end{proof}
\begin{cor}\label{ortwn}
An $L$-space $X$ is \textbf{orthogonally} complemented in $\mathrm{B}(G)$, i.e. there is a complement $X^{\perp}$ such that any $f\in X^{\perp}$ satisfies $f \bot X$.
\end{cor}
Important examples of $L$-spaces are $\mathrm{A}(G)$, the \textbf{Fourier algebra} of the group $G$, $\mathrm{B}_{0}(G)$ and the Zafran ideal; they are actually $L$-ideals. Note here that it follows that the Zafran ideal is complemented -- this was not recognised even in the commutative setting. A nice example of an $L$-subalgebra is $\mathrm{B}(G) \cap \mathrm{AP}(G)$. The fact that all of these examples are $L$-spaces follows immediately from Proposition \ref{Prop:Lspaces} because they certainly are $G$-bi-invariant.

We introduce the definition generalising the notion of a singular measure for compact Abelian groups.
\begin{de}
Let $G$ be a discrete group. An element $f\in \mathrm{B}(G)$ is called \textbf{singular} if $f\bot\delta_{e}$ where $e$ is the identity element of $G$.
\end{de}
Using the fact that $\mathrm{A}(G)$ is a closed $G$-invariant subspace of $\mathrm{B}(G)$ generated by $\delta_e$, we get the following theorem by Corollary \ref{ortwn}.
\begin{thm}\label{ort}
Let $G$ be a discrete group. The following holds true:
\begin{enumerate}[{\normalfont (i)}]
  \item If $h\in \mathrm{B}(G)$ is a singular element then $h\bot \mathrm{A}(G)$.
  \item There exists a closed linear subspace $B_{s}(G)$ of $\mathrm{B}(G)$ such that $\mathrm{B}(G)=B_{s}(G)\oplus \mathrm{A}(G)$. Moreover, this decomposition is orthogonal in the following sense: if $f\in \mathrm{B}(G)$, $f=f_{1}+f_{2}$, $f_{1}\in B_{s}(G)$ and $f_{2}\in \mathrm{A}(G)$ then $f_{1}\bot f_{2}$ and $\|f\|=\|f_{1}\|+\|f_{2}\|$.
\end{enumerate}
\end{thm}
The second assertion is known and was proved by G. Arsac in \cite{ar}. Some characterisation of elements from $B_{s}(G)$ and $\mathrm{A}(G)$ was given by T. Miao in \cite{mi} refined by E. Kaniuth, A. Lau and G. Schlichting in \cite{kls}.
\subsection{Further results on \texorpdfstring{$\mathrm{B}_{0}(G)\cap \mathrm{NS}(G)$}{the intersection of B0(G) with \mathrm{NS}(G)}}
Equipped with all the tools needed we will prove the criterion for non-naturality of the spectrum (see Theorem 3.6 in \cite{Zafran}).
\begin{thm}\label{nz}
Let $f\in \mathrm{B}_{0}(G)\setminus\{0\}$ be such that $f^{n}$ is singular for each $n\in\mathbb{N}$. Then $f\notin \mathrm{B}_{0}(G)\cap \mathrm{NS}(G)$.
\end{thm}
\begin{proof}
Suppose, towards the contradiction that $f\in \mathrm{B}_{0}(G)\cap \mathrm{NS}(G)$. The quotient algebra $\left(\mathrm{B}_{0}(G)\cap \mathrm{NS}(G)\right)/\mathrm{A}(G)$ has empty maximal ideal space by item 3. of Theorem \ref{glz}. Using the spectral radius formula and the definition of the quotient norm we obtain
\begin{equation*}
\lim_{n\rightarrow\infty}\left(\inf\{\|f^{n}+g\|:g\in \mathrm{A}(G)\}\right)^{\frac{1}{n}}=0.
\end{equation*}
But, by Theorem \ref{ort} and Fact \ref{norm} we have $\|f^{n}+g\|\geq \|f^{n}\|+\|g\|\geq \|f^{n}\|$ for every $g\in \mathrm{A}(G)$. This shows that spectral radius of $f$ in $\mathrm{B}(G)$ is equal to zero so since $\mathrm{B}(G)$ is semisimple we obtain $f=0$ contradicting the assumption.
\end{proof}
In order to proceed further, we need to investigate the supports of products of elements of $\mathrm{B}(G)$. We will recall now how the product of two elements of $\mathrm{B}(G)$ can be viewed as a (normal) functional on $C^{\ast}(G)$ ($W^{\ast}(G)$). There is a unitary representation of $G$ given by $G \ni g \mapsto u_g \otimes u_g \in C^{\ast}(G) \otimes C^{\ast}(G)\subset W^{\ast}(G) \otimes W^{\ast}(G)$, where $u_g$ is the element of the group viewed as sitting inside $C^{\ast}(G)$. When we complete the tensor products to obtain a $C^{\ast}$-algebra (or von Neumann algebra, in the case $W^{\ast}(G) \otimes W^{\ast}(G)$), we get an injective $\ast$-homomorphism $\Delta: C^{\ast}(G) \to C^{\ast}(G) \otimes C^{\ast}(G)$, produced by the universal property of $C^{\ast}(G)$. If we view $\Delta$ as having values in $W^{\ast}(G) \otimes W^{\ast}(G)$, then the universal property of $W^{\ast}(G)$ gives us an extension $\Delta: W^{\ast}(G) \to W^{\ast}(G) \otimes W^{\ast}(G)$. Now, if $f,g \in \mathrm{B}(G)$ then $(fg)(x) := (f\otimes g)(\Delta(x))$ for $x \in W^{\ast}(G)$. Using this, we can prove the following proposition.
\begin{prop}\label{wzornos}
Let $f,g \in \mathrm{B}(G)$ be positive definite. Then
\begin{equation}\label{form:support}
\mathrm{ s}(fg)=\inf\{p \in \mathcal{P}(W^{\ast}(G)): \Delta(p) \geqslant \mathrm{ s}(f) \otimes \mathrm{ s}(g)\},
\end{equation}
where $\mathcal{P}(W^{\ast}(G))$ is the lattice of projections in $W^{\ast}(G)$.
\end{prop}
\begin{proof}
By rescaling, we may assume that $f$ and $g$ are normalized, i.e. $f(e)=g(e)=1$.

Denote the right-hand side of the formula (\ref{form:support}) by $q$. We will show that $q \leqslant \mathrm{ s}(fg)$ and $\mathrm{ s}(fg)\leqslant q$. Let us start with the inequality $q\leqslant \mathrm{ s}(fg)$. By definition of the support, we get $1=(fg)(\mathrm{ s}(fg)) = (f \otimes g)(\Delta(\mathrm{ s}(fg)))$. One can check that $\mathrm{ s}(f\otimes g) = \mathrm{ s}(f)\otimes \mathrm{ s}(g)$, so we must have $\Delta(\mathrm{ s}(fg)) \geqslant \mathrm{ s}(f) \otimes \mathrm{ s}(g)$. By definition, $q$ is the smallest projection that satisfies $\Delta(q)\geqslant \mathrm{ s}(f) \otimes \mathrm{ s}(g)$, so $\mathrm{ s}(fg) \geqslant q$. To prove the reverse inequality, just note $\mathrm{ s}(fg)$ is the smallest projection such that $fg(\mathrm{ s}(fg))=1$ and any projection $q'$ satisfying $\Delta(q') \geqslant \mathrm{ s}(f) \otimes \mathrm{ s}(g)$ gives $fg(q')=1$.
\end{proof}
\begin{rem}
This proposition holds verbatim for normal states on Hopf--von Neumann algebras; the only property used is the existence of the $\ast$-homomorphism $\Delta$ implementing the multiplication of functionals.
\end{rem}
We can draw a nice conclusion from this formula.
\begin{cor}\label{przen}
Let $f_1$, $f_2$, $g_1$, and $g_2$ be positive definite and satisfy $\mathrm{ s}(g_1) \leqslant \mathrm{ s}(f_1)$ and $\mathrm{ s}(g_2) \leqslant \mathrm{ s}(f_2)$. Then $\mathrm{ s}(g_1 g_2) \leqslant \mathrm{ s}(f_1 f_2)$, so $g_1 g_2 \ll f_1 f_2$. In particular, if $f_1$ and $f_2$ are GNS faithful and $g_1\ll f_1$ and $g_2 \ll f_2$ then $g_1 g_2 \ll f_1 f_2$. What is more, if $f_1$ and $f_2$ are GNS faithful then also $f_1 f_2$ is GNS faithful, so this conclusion can be iterated, i.e. if we have $g_1,\dots, g_n$ satisfying $g_i \ll f_i$, where $f_i$ is GNS faithful, then $g_1\cdots g_n \ll f_1\cdots f_n$.
\end{cor}
\begin{proof}
The only nontrivial thing to check is that $f_1 f_2$ remains GNS faithful, if $f_1$ and $f_2$ were GNS faithful. Note that in this case the support of $f_1 f_2$, which we denote by $q$, is the smallest projection such that $\Delta(q) \geqslant \mathrm{ zs}(f_1) \otimes \mathrm{ zs}(f_2)$. Let $u\in W^{\ast}(G)$ be a unitary and let us multiply both sides of the inequality by $\Delta(u)$ on the left and $\Delta(u)^{\ast}$ on the right, resulting in
$$
\Delta(uqu^{\ast}) \geqslant \Delta(u) (\mathrm{ zs}(f_1) \otimes \mathrm{ zs}(f_2)) \Delta(u)^{\ast}.
$$
Since $\mathrm{ zs}(f_1) \otimes \mathrm{ zs}(f_2)$ belongs to the center of $W^{\ast}(G)\otimes W^{\ast}(G)$, the right-hand side just equals $\mathrm{ zs}(f_1) \otimes \mathrm{ zs}(f_2)$, hence we obtain $uqu^{\ast} \geqslant q$. But $u^{\ast}$ is also a unitary, so we also get $u^{\ast}q u \geqslant u$, i.e. $q \geqslant u q u^{\ast}$, so $uqu^{\ast}=q$. It means that $uq=qu$, so $q$ belongs to the center of $W^{\ast}(G)$.
\end{proof}
These statements are not fully satisfactory, because one might have hoped for the following implication to hold: if $f_1$ and $f_2$ are positive definite and $g_1 \ll f_2$ and $g_2 \ll f_2$, then $g_1 g_2 \ll f_1 f_2$ (note that this property holds for Abelian groups - see for example Appendix A.6 in \cite{grmc}). This is, unfortunately, not true, as the following example shows.
\begin{prop}
Let $\pi: S_3 \to \mathrm{M}_2$ be the standard, two-dimensional, irreducible representation of the permutation group $S_3$. Let $f(g) = \mathrm{Tr}(\pi(g))$ and $h_1(g):= \langle e_1, \pi(g) e_1 \rangle$, where $e_1$ is the first basis vector of the canonical basis of $\mathbb{C}^2$ on which $S_3$ acts. Then $f \ll h$ but it is not true that $f^2 \ll h^2$.
\end{prop}
\begin{proof}
Introduce $h_2(g):=\langle e_2, \pi(g) e_2\rangle$. We get $f=h_1+h_2$, therefore $f^2=h_1^2 + 2h_1 h_2 + h_2^2$. Since $h_i h_j(g) = \langle e_i \otimes e_j, (\pi(g)\otimes \pi(g)) (e_i \otimes e_j)\rangle$ for $i,j=1,2$, it is not hard to verify that the GNS representation associated with $f^2$ contains the tensor square $\pi^{\otimes 2}$, which decomposes into irreducibles as follows: $\pi^{\otimes 2} = \mathds{1} \oplus \mathrm{ sgn} \oplus \pi$, where $\mathds{1}$ is the one-dimensional trivial representation and $\mathrm{ sgn}$ is the one-dimensional sign representation: it is a faithful representation of $\mathbb{C}[S_3]$, hence the central support of $f^2$ is equal to the identity. On the other hand, the representation associated with $h_1^2$ is the \emph{symmetric} tensor square, which is a three-dimensional representation that decomposes as $\mathds{1} \oplus \pi$; its central support is orthogonal to the central support of the sign function, hence is strictly smaller than the central support of $f^2$.
\end{proof}
We will now apply our results about supports of products to obtain a criterion for being singular to the Zafran ideal.
\begin{prop}\label{singzaf}
Let $G$ be a discrete group and let $f\in \mathrm{B}(G)$ be positive, GNS faithful and such that for every $n\in\mathbb{N}$, $f^{n}$ is singular. Then $f\bot \mathrm{B}_{0}(G)\cap \mathrm{NS}(G)$.
\end{prop}
\begin{proof}
Suppose, to the contrary, that there exists $g\in \mathrm{B}_{0}(G)\cap \mathrm{NS}(G)$ such that $f$ is not singular with respect to $g$. We have $|g|=u^{\ast}g \in \mathrm{B}_{0}(G)\cap \mathrm{NS}(G)$, since $\mathrm{B}_0(G) \cap \mathrm{NS}(G)$ is an $L$-ideal. Using Lebesgue decomposition we are able to write $|g|=g_{1}+g_{2}$ where $g_{1}\ll f$ and $g_{2}\bot f$, $0\neq g_{1}$, and both $g_{1}$ and $g_{2}$ are positive definite. From positivity it follows that $g_1 \ll |g|$, so $g_{1}\in \mathrm{B}_{0}(G)\cap \mathrm{NS}(G)$. By Corollary \ref{przen} we have $g_{1}^{n}\ll f^{n}$ for every $n\in\mathbb{N}$. But, since each $f^{n}$ is singular we get $g_{1}^{n}$ is singular for every $n\in\mathbb{N}$ which contradicts Theorem \ref{nz}.
\end{proof}
\section{An approach to the Gelfand space of \texorpdfstring{$\mathrm{B}(G)$}{\mathrm{B}(G)}}\label{approach1}
In this section we will establish a description of elements of $\triangle(\mathrm{B}(G))$ using a non-commutative counterpart of the so-called \textbf{generalised characters} introduced (in the commutative case) by Schreider \cite{schreider} and used extensively later by many authors in investigations of spectral properties of measures. As we are going to use Proposition \ref{Prop:density}, where the separability of $A=C^{\ast}(G)$ is assumed, we will restrict our attention to second countable groups $G$ (note that that the Haar measure on a second countable locally compact group is $\sigma$-finite, therefore $L^{1}(G)$ is separable, hence so is its completion $C^{\ast}(G)$) .

For a positive definite $f\in \mathrm{B}(G)$ we define $\mathrm{A}(f):=\{g\in \mathrm{B}(G):g\ll f\}$. $\mathrm{A}(f)$ is a closed linear subspace of $\mathrm{B}(G)$. Moreover, if $f^{2}\ll f$ and $f$ is GNS faithful then Lemma \ref{przen} shows that $\mathrm{A}(f)$ is a closed subalgebra of $\mathrm{B}(G)$. Also, $\mathrm{A}(f)^{\ast}$ is naturally von Neumann algebra consisting of restrictions of elements from $W^{\ast}(G)$ to $\mathrm{A}(f)$, which may be identified with $\mathrm{vN}(|f|)$. Let $\varphi\in \mathrm{B}(G)^{\ast}$. Then clearly $\varphi|_{\mathrm{A}(f)}\in \mathrm{A}(f)^{\ast}$ and a unique element of $x\in \mathrm{A}(f)^{\ast}$ satisfying $\varphi(g)=x(g)$ for all $g\in \mathrm{A}(f)$ will be denoted by $\varphi_{f}$. There is also an obvious \textbf{compatibility condition} on restrictions to $\mathrm{A}(f_{1})$ and $\mathrm{A}(f_{2})$, i.e.
\begin{equation}\label{komp}
\text{If }f_{1},f_{2}\in \mathrm{B}(G)\text{ are positive definite and }f_{1}\ll f_{2}\text{ then }\varphi_{f_{2}}|_{\mathrm{A}(f_{1})}=\varphi_{f_{1}}.
\end{equation}
In addition, we see also
\begin{equation}\label{normm}
\sup\{\|\varphi_{f}\|_{\mathrm{A}(f)^{\ast}}\: f\text{ -- positive definite}\}=\|\varphi\|_{\mathrm{B}(G)^{\ast}}.
\end{equation}
The following fact will show that the reverse procedure is also available.
\begin{f}\label{func}
Let $\{\varphi_{f}\}_{f}$ be a family of elements such that $\varphi_{f}\in \mathrm{A}(f)^{\ast}$ indexed by all positive definite functions on $\mathrm{B}(G)$ that compatibility condition \eqref{komp} and the norm condition, which means the finiteness of the left-hand side of \eqref{normm}. Then putting $\varphi(f):=\varphi_{|f|}(f)$ for all $f\in \mathrm{B}(G)$ we obtain an element $\varphi\in \mathrm{B}(G)^{\ast}$.
\end{f}
\begin{proof}
The homogeneity is trivial by the compatibility condition as the multiplication by a scalar preserves the support. In order to prove linearity let us take $f_{1},f_{2}\in \mathrm{B}(G)$ and write:
\begin{equation*}
\varphi(f_{1}+f_{2})=\varphi_{|f_{1}+f_{2}|}(f_{1}+f_{2})=\varphi_{|f_{1}+f_{2}|}(f_{1})+\varphi_{|f_{1}+f_{2}|}(f_{2}),
\end{equation*}
now we use the fact $|f_{1}+f_{2}|\ll |f_{1}|+|f_{2}|$ and the compatibility condition to obtain
\begin{align*}
\varphi_{|f_{1}+f_{2}|}(f_{1})+\varphi_{|f_{1}+f_{2}|}(f_{2})&=\varphi_{|f_{1}|+|f_{2}|}(f_{1})+\varphi_{|f_{1}|+|f_{2}|}(f_{2}) \\
&=\varphi_{|f_{1}|}(f_{1})+\varphi_{|f_{2}|}(f_{2})=\varphi(f_{1})+\varphi(f_{2}).
\end{align*}
The norm condition ensures the boundedness of $\varphi$, which finishes the proof.
\end{proof}
In case $\varphi\in\triangle(\mathrm{B}(G))$ we have by Lemma \ref{przen} the following \textbf{multiplicative property} for $f,g,\omega\in \mathrm{B}(G)$, $\omega$ positive definite and \emph{GNS faithful}, $f,g\ll\omega$ and $\omega^{2}\ll \omega$:
\begin{equation}\label{multi}
\varphi_{\omega}(fg)=\varphi_{\omega}(f)\varphi_{\omega}(g).
\end{equation}
We will see that adding property \eqref{multi} to the assumptions of Fact \ref{func} we get an element of $\triangle(\mathrm{B}(G))$. Before proving that this is so, we need a simple lemma.
\begin{lem}\label{Lem:exp}
Suppose that $f\in \mathrm{B}(G)$ is positive and GNS faithful. Then $\exp(f)$ is also GNS faithful.
\end{lem}
\begin{proof}
It follows from Lemma \ref{przen} that $f^{n}$ is GNS faithful for any $n\in \mathbb{N}$. Since $\exp(f) = \sum_{n=0}^{\infty} \frac{f^{n}}{n!}$, it is clear that $\mathrm{ s}(\exp(f))= \bigvee_{n\in \mathbb{N}} \mathrm{ s}(f^{n})$. By our assumptions, this is equal to $\bigvee_{n\in \mathbb{N}} \mathrm{ zs}(f^{n})$. Since the supremum of a family of projections belongs to the von Neumann algebra generated by this family, this supremum belongs to the center of $W^{\ast}(G)$, therefore $\mathrm{ s}(\exp(f))=\mathrm{ zs}(\exp(f))$, i.e. $\exp(f)$ is GNS faithful.
\end{proof}
\begin{prop}\label{multip}
Let $\{\varphi_{f}\}_{f}$ be a family of elements such that $\varphi_{f}\in A(f)^{\ast}$ indexed by all positive definite functions on $\mathrm{B}(G)$ satisfying compatibility condition \eqref{komp}) and the norm condition which means the finiteness of the left hand side of \eqref{normm} and let us also assume multiplicative property \eqref{multi}. Then putting $\varphi(f):=\varphi_{|f|}(f)$ for all $f\in \mathrm{B}(G)$ we obtain an element $\varphi\in\triangle(\mathrm{B}(G))$.
\end{prop}
\begin{proof}
From Fact $\ref{func}$ we already know that $\varphi\in \mathrm{B}(G)^{\ast}$ so let us take $f,g\in \mathrm{B}(G)$ and consider $|f|+|g|$. Using Proposition \ref{Prop:density}, we find positive and GNS faithful $h$ such that $\mathrm{ zs}(h)=\mathrm{ zs}(|f|+|g|)$; in particular, both $f$ and $g$ are absolutely continuous with respect to $h$. Put now $\omega:=\exp(h)$. Then $f,g\ll\omega$ and $\omega^{2}\ll\omega$. Applying consecutively Lemma \ref{przen} ($|fg| \ll \omega^2 \ll \omega$), compatibility condition, \ref{multi} and compatibility condition again, we have
\begin{equation*}
\varphi(fg)=\varphi_{|fg|}(fg)=\varphi_{\omega}(fg)=\varphi_{\omega}(f)\varphi_{\omega}(g)=\varphi_{|f|}(f)\varphi_{|g|}(g)=\varphi(f)\varphi(g).
\end{equation*}
\end{proof}
\begin{lem}\label{jed}
Let $\varphi$ be a faithful state on a von Neumann algebra $\mathsf{M}$ and suppose that there exists $x\in\mathsf{M}$ such that $\|x\|=1$ and $\varphi(x)=1$. Then $x=\mathbf{1}$.
\end{lem}
\begin{proof}
We have $\varphi(x^{\ast})=1$, so $y=\frac{x+x^{\ast}}{2}$ satisfies $\|y\|\leqslant 1$ and $\varphi(y)=1$. We conclude that $\mathbf{1}-y \geqslant 0$ and $\varphi(\mathbf{1}-y)=0$, so we get $y=\mathbf{1}$ by faithfulness of $\varphi$. From equality $x = 2\mathbf{1}-x^{\ast}$ it follows that $x$ and $x^{\ast}$ commute, so they generate a commutative $C^{\ast}$-algebra; therefore they may be treated as complex-valued functions on a compact space. Since the norms of these functions are equal to $1$, we readily conclude from $x+x^{\ast}=2\mathbf{1}$ that $x=x^{\ast}=\mathbf{1}$. This is a standard proof showing that the unit $\mathbf{1}$ is an extreme point of the unit ball in any $C^{\ast}$-algebra.
\end{proof}
In order to establish a sufficient condition for belonging $\varphi\in\triangle(\mathrm{B}(G))$ to the Shilov boundary of $B(G)$ we will use the following standard fact (check Theorem 15.3 in \cite{z}).
\begin{f}
Let $A$ be a commutative Banach algebra. Then $\varphi\in\triangle(A)$ belongs to $\partial(A)$ if and only if for every open neighbourhood $U$ of $\varphi$ there exists $x\in A$ such that $\|\widehat{x}|_{\triangle(A)\setminus U}\|_{\infty}<\|\widehat{x}|_{U}\|_{\infty}$.
\end{f}
The last theorem of this section for commutative groups is due to Brown and Moran \cite{bm2}, see also Proposition 5.1.3 in \cite{grmc}. As we remarked before if $f\in \mathrm{B}(G)$ is positive definite and normalised then it can be viewed as a state on $\mathsf{vN}(f)=\mathrm{A}(f)^{\ast}$ which we will use in the proof.
\begin{thm}\label{brzsz}
Let $\varphi\in\triangle(\mathrm{B}(G))$ be such that for each positive definite function $f\in \mathrm{B}(G)$ the element $\varphi_{f}$ is a unitary. Then $\varphi\in\partial(A)$.
\end{thm}
\begin{proof}
By fact preceding the theorem it is enough to prove that if $U$ is neighbourhood of $\varphi$ then there is $f\in \mathrm{B}(G)$ satisfying $\widehat{f}(\varphi)=\varphi(f)=1$ and $|\widehat{f}|<1$ outside $U$. Using the definition of the weak$^{\ast}$ topology on $\triangle(\mathrm{B}(G))$ there exist $n\geq 1$, $\varepsilon>0$ and $f_{1},\ldots,f_{n}\in \mathrm{B}(G)$ such that the open set $V=\{\rho:|\widehat{f_{j}}(\rho)-\widehat{f_{j}}(\varphi)|<\varepsilon, \text{ for } 1\leq j\leq n\}$ is contained in $U$. Since any element of $\mathrm{B}(G)$ is a linear combination of four positive ones, by adjusting $\varepsilon$ we may assume that $f_j$'s are positive. Using Proposition \ref{Prop:density}, we may also require that they are GNS faithful. Let us define an element $g\in \mathrm{B}(G)$ by the formula
\begin{equation*}
g=a\prod_{j=1}^{n}\left[\varphi_{f_{j}}^{\ast}f_{j}+\left(\varphi_{f_{j}}^{\ast}f_{j}\right)\cdot\left(\varphi_{f_{j}}^{\ast}f_{j}\right)\right].
\end{equation*}
We have (recall that $\varphi_{f_j}^{\ast} f_{j} \ll f_{j}$)
\begin{align*}
\varphi(g)&=\prod_{j=1}^{n}\left(\varphi(\varphi_{f_{j}}^{\ast}f_{j})+\varphi(\varphi_{f_{j}}^{\ast}f_{j})\cdot\varphi(\varphi_{f_{j}}^{\ast}f_{j})\right) \\
&=\prod_{j=1}^{n}\left(\varphi_{f_{j}}(\varphi_{f_{j}}^{\ast}f_{j})+\varphi_{f_{j}}(\varphi_{f_{j}}^{\ast}f_{j})\cdot\varphi_{f_{j}}(\varphi_{f_{j}}^{\ast}f_{j})\right).
\end{align*}
We can continue this computation, obtaining
\begin{equation*}
\prod_{j=1}^{n}\left(f_{j}(\varphi_{f_{j}}^{\ast}\varphi_{f_{j}})+f_{j}(\varphi_{f_{j}}^{\ast}\varphi_{f_{j}})\cdot f_{j}(\varphi_{f_{j}}^{\ast}\varphi_{f_{j}})\right)
=\prod_{j=1}^{n}\left(\|f_{j}\|+\|f_{j}\|^{2}\right).
\end{equation*}
On the other hand, $|\varphi(g)|\leq \|g\|\leq \prod_{j=1}^{n}\left(\|f_{j}\|+\|f_{j}\|^{2}\right)$ which gives $\varphi(g)=\|g\|$, so there exists $a>0$ such that $f:=ag$ satisfies $\varphi(f)=\|f\|=1$. For any $\rho\in\triangle(\mathrm{B}(G))$ we obtain
\begin{equation*}
\rho(f)=a\prod_{j=1}^{n}\left(f_{j}(\varphi_{f_{j}}^{\ast}\rho_{f_{j}})+f_{j}(\varphi_{f_{j}}^{\ast}\rho_{f_{j}})\cdot f_{j} (\varphi_{f_{j}}^{\ast}\rho_{f_{j}})\right).
\end{equation*}
Hence
\begin{equation}\label{rowp}
|\rho(f)|\leq a\prod_{j=1}^{n}\left(\|f_{j}\|+\|f_{j}\|^{2}\right)
\end{equation}
and if the equality is attained then $|f_{j}(\varphi_{f_{j}}^{\ast}\rho_{f_{j}})|=\|f_{j}\|$. Using Lemma \ref{jed} to functionals $\frac{f_{j}}{\|f_{j}\|}$ we see that there constants $c_{j}$ of modulus $1$ such that $\varphi_{f_{j}}^{\ast}\rho_{f_{j}}=c_{j}\mathbf{1}$ and since $\varphi_{f_{j}}$ is a unitary we get $\rho_{f_{j}}=c_{j}\varphi_{f_{j}}$. Let us assume that we actually have the equality in \ref{rowp} for some $\rho\in\triangle(\mathrm{B}(G))$. Then
\begin{equation*}
\rho(f)=a\prod_{j=1}^{n}\left(c_{j}\|f_{j}\|+c_{j}^{2}\|f_{j}\|^{2}\right).
\end{equation*}
An elementary exercise in complex numbers shows that $c_{j}=1$ for $j=1,\ldots,n$. We have proved the following implication: if $\rho\in\triangle(\mathrm{B}(G))$ is such that $|\rho(f)|=1$ then $\rho\in V$. Since $f$ is of norm one we conclude $|\widehat{f}|<1$ outside $V$ and of course the same statement holds outside $U$.
\end{proof}
\begin{rem}
In fact, relying on Theorem 1 in \cite{wa2}, one can prove more. First, note that $W^{\ast}(G) \hookrightarrow \prod_{z} z W^{\ast}(G)$, where the product is indexed by central projections in $W^{\ast}(G)$ that are equal to central support of some positive definite function of $G$; this map is an embedding since any representation decomposes as a direct sum of cyclic representations, which correspond to positive definite functions. Second, our condition of unitarity of generalised characters means exactly that the image of a character under this embedding is unitary, therefore the character itself is a unitary element of $W^{\ast}(G)$; by Theorem 1 in \cite{wa2}, it means that it is an evaluation at some point $g\in G$, therefore belongs to the Shilov boundary of $\mathrm{B}(G)$. On the other hand, our proof is completely elementary, and we believe that generalised characters, which are customary to use in classical harmonic analysis, may shed some light on unresolved issues in the non-commutative world.
\end{rem}
\section{Free groups}\label{free groups}
In this section we will examine some examples of positive definite functions in terms of their spectral properties using the results from section \ref{Section Zafran}.

First of all, let $\mathbb{F}_{k}$ denote the free group on $k$ generators ($1\leq k\leq\infty)$ and let us denote by $|x|$ the length of the reduced word $x\in \mathbb{F}_{k}$. It is evident that $\mathbb{F}_{k}$ contains many infinite Abelian subgroups and these groups are maximally almost periodic as they are residually finite, so all facts proved in sections \ref{basic results} and \ref{decomposition} are valid for $\mathbb{F}_{k}$.

Our first example is the Haagerup function (see \cite{haa}) defined by the formula $f_{r}(x):=r^{|x|}$ for $0<r<1$ and $x\in \mathbb{F}_{k}$. Clearly, $f_{r}\in c_{0}(\mathbb{F}_{k})$ and $f_{r}^{m}=f_{r^{m}}$ for every $m\in\mathbb{N}$. Positive definiteness of $f_{r}$ was proved in \cite{haa}. Let us compute the $\ell_{2}$ norm of $f_{r}$:
\begin{align*}
\|f_{r}\|_{\ell_{2}}^{2}&=\sum_{x\in F_{k}}r^{2|x|}=\sum_{n=1}^{\infty} \left|\{x\in F_{k}:|x|=n\}\right|r^{2n} \\
&=2k\sum_{n=1}^{\infty}(2k-1)^{n-1}r^{2n}=\frac{2k}{2k-1}\sum_{n=1}^{\infty}\left((2k-1)r^{2}\right)^{n}.
\end{align*}
The latter series is convergent iff $(2k-1)r^{2}<1$ implying that $f_{r}^{m}\in \ell_{2}(\mathbb{F}_{k})$ iff $m>\frac{\mathrm{ln}(2k-1)}{2\mathrm{ln}\frac{1}{r}}$. In particular, for every $r\in (0,1)$ there exists $m\in\mathbb{N}$ such that $f_{r}^{m}\in \ell_{2}(\mathbb{F}_{k})$. Since we always have $\ell_{2}(G)\subset \mathrm{A}(G)\subset \mathrm{B}_{0}(G)\cap \mathrm{NS}(G)$, for an arbitrary discrete group $G$, Theorem \ref{glz} or even elementary Proposition \ref{zawpod} gives $f_{r}\in \mathrm{B}_{0}(G)\cap \mathrm{NS}(G)$.
\begin{cor}
For every $r\in (0,1)$ the Haagerup function $f_{r}(x)=r^{|x|}$ has a natural spectrum.
\end{cor}
So far we established that for $r< \frac{1}{\sqrt{2k-1}}$ the function $f_r$ belongs to the Fourier algebra $A(\mathbb{F}_k)$. If a sequence $(r_k)_{k\in\mathbb{N}}$ converges to $r$ then the sequence $(f_{r_k})_{k \in \mathbb{N}}$ converges pointwise to $f_r$. Since these functions are positive definite and normalised, they are uniformly bounded in $B(\mathbb{F}_k)$, hence this is even weak$^{\ast}$-convergence. It follows that $f_{\scaleobj{0.7}{\frac{1}{\sqrt{2k-1}}}}$ belongs $B_{\lambda}(\mathbb{F}_k)$, the weak$^{\ast}$-closure of $A(\mathbb{F}_k)$ inside $B(\mathbb{F}_k)$, which can be also identified with the dual of $C^{\ast}_{\lambda}(\mathbb{F}_k)$ -- the reduced $C^{\ast}$-algebra of the free group. We do not know much about the behaviour beyond this regime. It was already noted by Haagerup (cf. \cite[Corollary 3.2]{haa}) that for $r> \frac{1}{\sqrt{2k-1}}$ does not belong to $B_{\lambda}(\mathbb{F}_k)$. The argument goes as follows. Suppose that $f_r \in B_{\lambda}(\mathbb{F}_k)$. This means that it defines a bounded linear functional on $C^{\ast}_{\lambda}(\mathbb{F}_k)$. Let $\chi_n$ be the characteristic function of the set of words of length $n$. It follows from the celebrated Haagerup inequality (\cite[Lemma 1.4]{haa}) that $\|\chi_n\|_{\scaleobj{0.7}{C^{\ast}_{\lambda}(\mathbb{F}_k)}} \leqslant (n+1)\|\chi_n\|_{\scaleobj{0.7}{\ell_2(\mathbb{F}_k}}$ and it is easy to compute that $\|\chi_n\|^2_{\scaleobj{0.7}{\ell_2(\mathbb{F}_k)}} = 2k(2k-1)^{n-1}$, because the number of words of length $n$ is equal to $2k(2k-1)^{n-1}$. We also have $f_r(\chi_n) = r^{n} 2k(2k-1)^{n-1}$. If $f_r$ defined a bounded functional on $C^{\ast}_{\lambda}(\mathbb{F}_{k})$, there would exist a positive constant $C$ such that $|f_r(\chi_n)| \leqslant C \|\chi_n\|_{\scaleobj{0.7}{C^{\ast}_{\lambda}(\mathbb{F}_k)}}$. It would force the inequality $r^{n} 2k(2k-1)^{n-1} \leqslant C(n+1) \left(2k(2k-1)^{n-1}\right)^{\frac{1}{2}}$, which may happen only for $r\leqslant \frac{1}{\sqrt{2k-1}}$.

A more general notion was introduced in \cite{am} and \cite{ftm}: let $F$ be a free subset of a discrete group $G$ (a subset with no relations among its elements), a function $u\colon G\rightarrow\mathbb{C}$ vanishing off the subgroup generated by $F$ is called a Haagerup function if:
\begin{itemize}
  \item $u(1)=1$  and $|u(x)|\leq 1$ for all $x\in G$,
  \item $\overline{u(x)}=u(x^{-1})$ for $x\in F$,
  \item $u(xy)=u(x)u(y)$ if $|xy|=|x|+|y|$.
\end{itemize}
As it is proved in both of the papers mentioned above, such a function is positive definite. Let us cite the second theorem from \cite{ftm} (recall that $\mathrm{B}_{\lambda}(G)$ is the dual of the reduced group C$^{\ast}$-algebra or equivalently the weak$^{\ast}$ closure of $\mathrm{A}(G)$ in $\mathrm{B}(G)$).
\begin{thm}
Let $F$ be a free subset of $G$ and $u$ a Haagerup function defined as above.
\begin{enumerate}
  \item If $\sum_{x\in F}|u(x)|^{2}=\infty$ then  $u\bot B_{\lambda}(G)$,
  \item If $\sum_{x\in F}|u(x)|^{2}\leq \frac{1}{2}$ then $u\in l^{2}(G)\subset A(G)$.
\end{enumerate}
\end{thm}
This result and Theorem \ref{glz} implies the following.
\begin{cor}
Let $F$ be a free subset of $G$ and $u\in \mathrm{B}_{0}(G)$ a Haagerup function as before.
\begin{equation*}
\text{If }\sum_{x\in F}|u(x)|^{n}=\infty\text{ for every }n\in\mathbb{N}
\end{equation*}
then $u$ does not have a natural spectrum.
\end{cor}

\subsection{Free Riesz products}
We will now follow the article \cite{marek} (and its notation) to define a non-commutative analogue of Riesz products on the free groups; this class will turn out to contain many examples of positive definite functions without a natural spectrum. Let us start with a definition.
\begin{de}
Let $(\Gamma_1, \varphi_1)$ and $(\Gamma_2,\varphi_2)$ be discrete groups equipped with normalised positive definite functions. Let $\Gamma = \free{\Gamma_1}{\Gamma_2}$ be the free product of groups. Using $\varphi_1$ and $\varphi_2$ we may define a new positive definite function $\varphi:=\free{\varphi_1}{\varphi_2}: \Gamma \to \mathbb{C}$. Fix $x\in \Gamma$. If $x=e$ then we put $\varphi(x)=1$, otherwise $x$ can be written as a product $x=\gamma_{1}\gamma_{2}\dots\gamma_{n}$, where $\gamma_k \in \Gamma_{i_k}\setminus \{e\}$ and $i_1\neq i_2\neq\dots \neq i_n$ -- in such a case $\varphi(x)=\varphi_{i_1}(\gamma_1)\varphi_{i_2}(\gamma_2)\dots \varphi_{i_n}(\gamma_n)$.
\end{de}
\begin{rem}
It is not trivial to prove that such a function $\varphi$ is positive definite.
\end{rem}
This notion of free product has been generalised and studied further in \cite{marekroland} and \cite{conditionallyfree}, resulting in the notion of $c$-freeness -- it is a generalisation of the notion of freeness introduced by Voiculescu. A version for free products of completely positive maps was provided by Boca in \cite{florin}, independently of previous developments.

This construction can be iterated, even infinitely many times, producing positive definite functions on free products of infinitely many groups. Bo\.{z}ejko utilised such infinite products to obtain a free counterpart of classical Riesz products. Our plan now is to recall this construction.

Let $\mathbb{F}_{\infty}= \bigcirc_{k=1}^{\infty} \mathbb{Z}^{(k)}$ be the free group on infinitely many generators, viewed as an infinite free product of copies of $\mathbb{Z}$; we will call the $k$-th free generator $x_k$. For each $k \in \mathbb{N}$ consider a positive definite function $v_k:= \delta_{e} + \alpha_k \delta_{x_k} + \overline{\alpha}_k \delta_{x_k^{-1}}$, where $0<|\alpha_k| \leqslant \frac{1}{2}$. The free product function $R:=\bigcirc_{k=1}^{\infty} v_k$ will be called a \textbf{free Riesz product}. We want to present now the main result concerning free Riesz products (cf. \cite[Corollary 2]{marek}). Before that, let us recall that $B_{\lambda}(G)$ is the weak$^{\ast}$ closure of $\mathrm{A}(G)$ inside $\mathrm{B}(G)$.
\begin{thm}
Let $v_k = \delta_e + \alpha_k \delta_{x_k} + \overline{\alpha}_k \delta_{x_k^{-1}}$ be a positive definite function on $\mathbb{Z}^{(k)}$ with $|\alpha_k| \leqslant \frac{1}{2}$. Form the free Riesz product $R=\bigcirc_{k=1}^{\infty} v_k$, which is a positive definite function on $\mathbb{F}_{\infty}=\bigcirc_{k=1}^{\infty} \mathbb{Z}^{(k)}$. Let $\beta:= \sum_{k=1}^{\infty} |\alpha_k|^2$ and $\gamma:=\sum_{k\neq l} |\alpha_{k}|^2 \cdot |\alpha_l|^2$. Then:
\begin{enumerate}[{\normalfont (i)}]
\item if $\beta<\infty$ and $\gamma<1$ then $R \in \ell_2(\mathbb{F}_{\infty})$;
\item if $\beta<\infty$ and $\gamma\leqslant 1$ then $R \in B_{\lambda}(\mathbb{F}_{\infty})$;
\item\label{singriesz} if $\beta=\infty$ then $R \perp B_{\lambda}(\mathbb{F}_{\infty})$;
\item if $R \in \mathrm{B}_{\lambda}(\mathbb{F}_{\infty})$ then $\beta \leqslant 2$.
\end{enumerate}
\end{thm}
Item \eqref{singriesz} from this theorem will be critical for applications. It allows us to find examples of positive definite functions that are singular, so it should be helpful in invoking Theorem \ref{nz} in order to find examples of functions with a non-natural spectrum. We want to have examples such that all powers are singular. Note that the free product construction interacts nicely with taking powers -- if $R$ is a free Riesz product then $R^{m}$ is also a free Riesz product, one only has to replace the numbers $\alpha_k$ with $\alpha_k^{m}$. Therefore previous theorem immediately implies the following corollary.
\begin{cor}\label{rnie}
Suppose that the sequence $(\alpha_k)_{k\in \mathbb{N}}$ satisfies $0<|\alpha_k| \leqslant \frac{1}{2}$, $\lim_{k\to\infty} \alpha_k=0$, and $\sum_{k=1}^{n} |\alpha_k|^{2m}=\infty$ for any $m\in \mathbb{N}$. Consider $v_k:=\delta_e + \alpha_k \delta_{x_k} + \overline{\alpha}_k \delta_{x_k^{-1}}$ and form the free Riesz product $R:=\bigcirc_{k=1}^{\infty} v_k$. Then all the powers of $R$ are singular.
\end{cor}
It is very simple to find examples of sequences $(\alpha_k)_{k\in\mathbb{N}}$ that satisfy the assumptions of this corollary, e.g. $\alpha_k = \frac{1}{2 \log(10+k)}$, so we are able to show examples of functions with a non-natural spectrum. Let us state it explicitly as a corollary.
\begin{cor}
There exist free Riesz products in $\mathrm{B}_{0}(\mathbb{F}_{\infty})$ with a non-natural spectrum.
\end{cor}
It is worth to note that we can prove the non-naturality of the spectrum of some Riesz products without the restriction $(\alpha_{k})\in c_{0}$. The presented argument is in the spirit of a Riesz product proof of the Wiener--Pitt phenomenon by C. C. Graham (check \cite{graham}) but to proceed we need first to establish some properties of the \textbf{coset ring} defined as a ring of sets generated by cosets i.e. the smallest collection of subsets containing all cosets and closed under finite unions, intersections and complementation. Observe that the characteristic function of a member $U$ of the coset ring can be written in the following form
\begin{equation}\label{pos}
\chi_{U}=\sum_{l=1}^{n_{1}}\chi_{A_{l}}-\sum_{k=1}^{n_{2}}\chi_{B_{k}}\text{ where }A_{l}\text{ and }B_{k}\text{ are cosets}.
\end{equation}
The coset ring of a group $G$ is closely related to the structure of idempotents in $B(G)$, more explicitly a function $f:G\rightarrow\{0,1\}$ is an idempotent in $B(G)$ iff the support of $f$ belongs to the coset ring - the case of $G=\mathbb{Z}$ was proved by H. Helson (\cite{hel}), for $G=\mathbb{Z}^{n}$ by W. Rudin (\cite{r1}), the general (Abelian) case was settled by P. Cohen (\cite{c}) and the non-commutative version (\cite{ho}) is due to B. Host (note that there are bounds in terms of the norm of the idempotent for the number of summands in (\ref{pos}) given by B. Green, T. Sanders in \cite{gs} for the Abelian case and by T. Sanders in \cite{sa} for non-Abelian groups).

We need the following proposition.
\begin{prop}\label{wars}
Let $G$ be a discrete group and let $U$ be an infinite member of the coset ring of $G$. Then $U$ almost contains an infinite coset (i.e. excluding finitely many elements).
\end{prop}
\begin{proof}
By $(\ref{pos})$ we are able to write $U=\bigcup_{l=1}^{n_{1}}A_{l}\setminus\bigcup_{k=1}^{n_{2}}B_{k}$ for some (non-empty) cosets $A_{l},B_{k}$, $l\in\{1,\ldots,n_{1}\}$, $k\in\{1,\ldots,n_{2}\}$ with $\bigcup_{k=1}^{n_{2}}B_{k}\subset \bigcup_{l=1}^{n_{1}}A_{l}$. The formulation of the proposition enables us to remove finite elements from $(\ref{pos})$ so we restrict the discussion to cosets of infinite subgroups. Also, it is clear that is enough to consider the case $n_{1}=1$ and we put $A:=A_{1}$. Let us start by showing the assertion for $n_{2}=1$. Then $U=A\setminus B$ with the cosets $A$ and $B$ satisfying $B\subset A$. Denoting $A=xH$ and $B=yH_{1}$ for $x,y\in G$ and $H,H_{1}$ - subgroups of $G$, we have (since the intersection of cosets is again a coset) $B=yH_{1}=yH_{1}\cap xH=z H\cap H_{1}$ for some $z\in G$. This implies $x^{-1}zH\cap H_{1}\subset H$ and of course $x^{-1}z\in H$. But $H$ is a disjoint unions of cosets of an infinite subgroup $H\cap H_{1}$ and so there exists $u\in H$ such that $uH\cap H_{1}\subset \left(H\setminus x^{-1}z H\cap H_{1}\right)$. Multiplying back by $x$ we get $xu H\cap H_{1}\subset A\setminus B$.

We move now to the general case. Elementary considerations similar to the ones presented above justify that it is enough to prove the assertion with $U$ of the following form
\begin{equation*}
U=H\setminus\bigcup_{k=1}^{n_{2}}a_{k}H_{k}=\bigcap_{k=1}^{n_{2}}\left(H\setminus a_{k}H_{k}\right),
\end{equation*}
where $H,H_{1},\ldots,H_{k}$ are (not necessarily distinct) subgroups of $G$ and $a_{k}\in H$ for $k\in\{1,\ldots,n_{2}\}$. We consider two cases.

If $H_{1}\cap H_{2}\cap\ldots\cap H_{n_{2}}$ is infinite then since $H$ is a union of cosets of each of the subgroups $H_{k}$ and the intersection of cosets is again a coset we immediately obtain the inclusion $hH_{1}\cap H_{2}\cap\ldots\cap H_{n_{2}}\subset U$ for some $h\in H$ which finishes the proof in this case.

In case of finite $H_{1}\cap H_{2}\cap\ldots\cap H_{n_{2}}$ we place one after another all identical subgroups (if there are any) and pick the biggest $l\in\{1,\ldots,n_{2}-1\}$ such that $H_{1}\cap\ldots\cap H_{l}$ is infinite. In the same way as before we see that there exists $h\in H$ for which
\begin{equation*}
hH_{1}\cap\ldots\cap H_{l}\subset \bigcap_{k=1}^{l}\left(H\setminus a_{k}H_{k}\right).
\end{equation*}
The choice of $l$ implies that $H_{1}\cap\ldots\cap H_{l}\cap H_{l+1}$ is finite. Now,
\begin{gather*}
hH_{1}\cap\ldots\cap H_{l}=hH_{1}\cap\ldots\cap H_{l}\cap\bigcap_{k=1}^{l}\left(H\setminus a_{k}H_{k}\right)\cap H=\\
=\left[hH_{1}\cap\ldots\cap H_{l}\cap\bigcap_{k=1}^{l}\left(H\setminus a_{k}H_{k}\right)\right]\cap \left[\left(H\setminus a_{l+1}H_{l+1}\right)\cup a_{l+1}H_{l+1}\right]=\\
=\left[hH_{1}\cap\ldots\cap H_{l}\cap\bigcap_{k=1}^{l+1}\left(H\setminus a_{k}H_{k}\right)\right]\cup \left[hH_{1}\cap\ldots\cap H_{l}\cap\bigcap_{k=1}^{l}\left(H\setminus a_{k}H_{k}\right)\cap a_{l+1}H_{l+1}\right].
\end{gather*}
Note that $hH_{1}\cap\ldots\cap H_{l}\cap\bigcap_{k=1}^{l}\left(H\setminus a_{k}H_{k}\right)\cap a_{l+1}H_{l+1}\subset \left(hH_{1}\cap\ldots\cap H_{l}\right)\cap a_{l+1}H_{l+1}$ and the latter set is either empty or is a coset of a finite subgroup $H_{1}\cap\ldots\cap H_{l}\cap H_{l+1}$. Consequently, $\left(hH_{1}\cap\ldots\cap H_{l}\right)\setminus X_{1}\subset \bigcap_{k=1}^{l+1}\left(H\setminus a_{k}H_{k}\right)$ for some finite set $X_{1}$. We perform this procedure $k-l$ times to finish the argument.
\end{proof}
\begin{thm}
The free Riesz product $R:=\bigcirc_{k=1}^{\infty}\left(\delta_e + \frac{1}{2}\delta_{x_k} + \frac{1}{2} \delta_{x_k^{-1}}\right)\in B(\mathbb{F}_{\infty})$ does not have a natural spectrum.
\end{thm}
\begin{proof}
Suppose towards a contradiction, that $\sigma(R)=R(\mathbb{F}_{\infty})$. Clearly,
\begin{equation}\label{risz}
R(\mathbb{F}_{\infty})=\{0\}\cup\left\{\frac{1}{2^{n}}\right\}_{n\in\mathbb{N}}\cup\{1\}\text{ and }R(x)=\frac{1}{2}\text{ iff }x=x^{\pm 1}_{k}\text{ for some }k\in\mathbb{N}.
\end{equation}
By $(\ref{risz})$ we are able to find two disjoint open sets $A,B\subset\mathbb{C}$ satisfying $\sigma(R)\subset A\cup B$ and $B\cap\sigma(R)=\{\frac{1}{2}\}$. Let $F:A\cup B\rightarrow\mathbb{C}$ be a function defined by the formula:
\begin{equation*}
F(z)=\left\{\begin{array}{c}
              0\text{ for }z\in A, \\
              1\text{ for }z\in B.
            \end{array}\right.
\end{equation*}
Then $F$ is holomorphic on $A\cup B$ and we apply the functional calculus to obtain an idempotent $F(R)$. The second part of $(\ref{risz})$ implies that the support of $F(R)$ is a set $\{x_{k},x_{k}^{-1}\}_{k\in\mathbb{N}}$. However, by the Host Idempotent Theorem (see \cite{ho}) the support of $F(R)$ has to belong to the coset ring of $\mathbb{F}_{\infty}$. By Proposition \ref{wars} the set $\{x_{k},x_{k}^{-1}\}_{k\in\mathbb{N}}$ almost contains an infinite coset. This is clearly impossible but let us present a formal proof for the reader's convenience. Suppose that $\left(gG\setminus X\right)\subset \{x_{k},x_{k}^{-1}\}_{k\in\mathbb{N}}$ for some infinite subgroup $G$, finite set $X$ and $g\in\mathbb{F}_{\infty}$. Then $G\subset \{g^{-1}x_{k},g^{-1}x_{k}^{-1}\}_{k\in\mathbb{N}}\cup g^{-1}X$. As the set $X$ is finite there is $k_{0}\in\mathbb{N}$ such that $(g^{-1}x_{k_{0}})^{-1}=x_{k_{0}}^{-1}g\in\{g^{-1}x_{k},g^{-1}x_{k}^{-1}\}_{k\in\mathbb{N}}$ so for some $k_{1}\in\mathbb{N}$ we have either $x_{k_{0}}^{-1}g=g^{-1}x_{k_{1}}$ or $x_{k_{0}}^{-1}g=g^{-1}x^{-1}_{k_{1}}$ and both cases lead easily to a contradiction.
\end{proof}
\section{Concluding remarks and open problems}
\begin{enumerate}
  \item The main open problem related to Section \ref{basic results} is the occurrence of the Wiener--Pitt phenomenon for discrete groups not containing any infinite Abelian subgroups. The groups with even stronger properties called \textbf{Tarski monsters} (infinite groups such that each of its non-trivial subgroup has $p$ elements for a fixed prime number $p$) were constructed by A. Yu. Olshanskii (see the book \cite{ol}) but the authors were unable to verify if such groups admit elements with a non-natural spectrum.
  \item The argument used in order to prove the extension of the theorem of Hatori and Sato (Theorem \ref{rozklad}) utilizes the special assumption on the group $G$ (it is supposed to be maximally almost periodic) but it is not known if any additional condition is needed to obtain a similar statement. Concretely, it is possible that $\mathrm{B}(G)=\mathrm{NS}(G)+\mathrm{NS}(G)$ for any discrete group $G$ (this problem is open even for $G=\mathbb{Z}$).
  \item Passing to Section \ref{Section Zafran}, the authors do not know if GNS faithfullness is necessary in Proposition \ref{singzaf}.
  \item It is clear that $r(\mathrm{A}(G))=\{f\in \mathrm{B}(G):f^{n}\in \mathrm{A}(G)\text{ for some }n\in\mathbb{N}\}\subset \mathrm{B}_{0}(G)\cap \mathrm{NS}(G)$ and there is a very nice argument showing that $r(\mathrm{A}(G))\varsubsetneq \mathrm{B}_{0}(G)\cap \mathrm{NS}(G)$ at least for commutative (discrete) $G$. Indeed, if $r(\mathrm{A}(G))=\mathrm{B}_{0}(G)\cap \mathrm{NS}(G)$ then $r(\mathrm{A}(G))$ is closed and thus it is a commutative Banach algebra. Since $\mathrm{A}(G)$ is an ideal in $r(\mathrm{A}(G))$ we can construct the quotient $A:=r(\mathrm{A}(G))/\mathrm{A}(G)$ which satisfies $\triangle(A)=\emptyset$ ($A$ is radical) as $\triangle(r(\mathrm{A}(G)))=\triangle(\mathrm{A}(G))=G$ and by definition of $r(\mathrm{A}(G))$ every element in $A$ is nilpotent ($A$ is nil). But a Baire category argument (see \cite{grab}) implies that $A$ is nilpotent, i. e. there exists $N\in\mathbb{N}$ such that $f^{N}=0$ for every $f\in A$. Now a suitable chosen Riesz product leads to a contradiction.
      However, the question if the equality $\overline{r(\mathrm{A}(G))}=\mathrm{B}_{0}(G)\cap \mathrm{NS}(G)$ holds true remains unanswered even for $G=\mathbb{Z}$.
  \item An element $f\in \mathrm{B}(G)$ has \textbf{independent powers} if $f^{n}\bot f^{m}$ for every $n\neq m$. The theorem of G. Brown and W. Moran (check \cite{bm}) asserts that the classical Riesz products satisfying the assumption of Corollary \ref{rnie} (even without the requirement of belonging to $c_{0}$) have independent powers. It is plausible that the same holds for free Riesz products but there is an obstacle in the direct transfer of the original argument - we do not know if free Riesz products are GNS faithful.
	\item Free Riesz products can be also used to produce group actions that are interesting from the point of view of group theory. Let $f \in c_{0}(G)$ be a real, positive definite function on a discrete group $G$ such that all of its powers are singular to $\mathrm{B}_{\lambda}(G)$; some free Riesz products on $\mathbb{F}_{\infty}$, as remarked in Corollary \ref{rnie}, have this property. Let $\pi: G \to \mathcal{O}(\mathsf{H})$ be the GNS representation associated with $f$, where $\mathsf{H}$ is a real Hilbert space; it is mixing since $f\in c_0(G)$. Then the GNS representation associated with $f^{n}$ may be identified with the $n$-fold \emph{symmetric} tensor power of $\pi$ and we can form the direct sum of these representations, which acts on the direct sum of all symmetric tensor powers of $\mathsf{H}$, i.e. the \textbf{symmetric Fock space} over $\mathsf{H}$. This setting always gives rise to a measure preserving action of $G$ -- the \textbf{Gaussian action} associated with $\pi$ -- such that the representation on the Fock space is naturally identified with the Koopman representation. Gaussian action associated with the left regular representation is the Bernoulli shift. Generally mixing Gaussian actions are viewed as a generalisation of Bernoulli actions. By our singularity assumption, none of subrepresentations of the representation on the Fock space is weakly contained in the left regular representation. Therefore, in our case the properties of the associated Gaussian action may be stated informally as follows: no part of this action looks like a Bernoulli action. It is an interesting example, although we believe that many occurrences of this phenomenon are familiar to group theorists.
  \item A Hermitian probability measure with independent powers on a locally compact Abelian group $G$ has a spectrum filling the whole unit disc (this is the theorem of W.~J. Bailey, G. Brown and W. Moran \cite{bbm}, see also Theorem 6.1.1 in \cite{grmc}). Once again it seems that this theorem should be correct also for non-commutative groups but this time the main problem lies in the lack of a notion of a topological support of an element from $\mathrm{B}(G)$.
  \item In \cite{ow1} it was shown that there exists a compact (uncountable) set $A\subset\mathbb{C}$ (called \textbf{Wiener--Pitt set}) such that every measure $\mu\in M(\mathbb{T})$ satisfying $\widehat{\mu}(\mathbb{Z})\subset A$ has a natural spectrum. The proof of the analogue for some non-commutative groups would be a tremendous progress.
	\item Even though we worked specifically with discrete groups, some of the results carry over to general locally compact groups. Let us say first, which results cannot be extended to the framework of locally compact groups.
	
To retain the duality between the Fourier--Stieltjes algebra and the group $C^{\ast}$-algebra, we need to consider only continuous positive definite functions, sometimes called functions of positive type. The topological issues make the situation vastly more complicated. Suppose that $H$ is a closed subgroup of a locally compact group $G$. Unless the group $H$ is also an open subgroup, our procedure of extending positive definite functions on $H$ to $G$ by putting value zero outside $H$ is not valid. We heavily relied on algebraic properties of this particular extension, so mere existence of an extension would not be enough for our purposes. It turns out that even this fails. An concrete example is as follows: let $\mathbb{R} \rtimes \mathbb{R}_{+}$ be the semidirect product of $(\mathbb{R},+)$ and $(\mathbb{R}_{+}, \cdot)$, which is sometimes called the $ax+b$ group. On the subgroup $\mathbb{R}$ we can consider a character $x \mapsto \exp(itx)$ for some $t\in\mathbb{R}$; this positive definite functions extends to a continuous definite function on $\mathbb{R} \rtimes \mathbb{R}_{+}$ only for $t=0$, see \cite[discussion on pages 204--205]{ey}. The results from this paper that relied on infinite Abelian subgroups used this feature of discrete groups in an essential way. Among them are Theorem \ref{wpf} and Theorem \ref{nieos}. On the other hand, Theorem \ref{rozklad} holds also for locally compact groups.

The proof of Theorem \ref{glz} relied on Lemma \ref{przelicz}, whose proof works only for discrete groups. This theorem cannot be extended to more general groups, even for Abelian groups (cf. \cite{hs}). On the other hand, all the material that follows this result, from Subsection \ref{mssandas} on, works equally well for general locally compact groups, as it only appealed to general theory of $C^{\ast}$-algebras and von Neumann algebras. In particular, Proposition \ref{Prop:translation}, Proposition \ref{Prop:Lspaces} and Theorem \ref{brzsz} are still true in this generalised setting. Note that any results pertaining to the Zafran ideal $\mathrm{B}_{0}(G) \cap \mathrm{NS}(G)$, such as Theorem \ref{nz}, are not valid, because we expect that $\mathrm{B}_{0}(G) \cap \mathrm{NS}(G)$ does not have any reasonable algebraic structure if the group $G$ is not discrete.
\item There is still not much known about properties of the Haagerup functions $f_r$ discussed at the beginning of Section \ref{free groups}. There are two specific problems that we would like to advertise: Is $f_{\scaleobj{0.7}{\frac{1}{\sqrt{2k-1}}}}$ singular to $\mathrm{A}(\mathbb{F}_k)$? Is $f_r$ singular to $B_{\lambda}(\mathbb{F}_k)$ for all $r> \frac{1}{\sqrt{2k-1}}$? If it is so, it would provide a nice example of a `phase transition' at point $r=\frac{1}{\sqrt{2k-1}}$ -- below it the functions belong to the Fourier algebra and above that level they are singular even to the reduced Fourier--Stieltjes algebra.
\end{enumerate}

\end{document}